\documentclass[12pt,reqno,twoside]{article}
\usepackage{amssymb,amsfonts,amsthm,amsmath}
\usepackage{graphicx}
\usepackage{color}
\usepackage{plain}
\usepackage[left=1 in,top=1 in,right=1 in,bottom=1 in]{geometry}

\numberwithin{equation}{section}

\def\rev{}

\def\beq{\begin{equation}}
\def\eeq{\end{equation}}
\def\beqs{\begin{equation*}}
\def\eeqs{\end{equation*}}

\newtheorem{theorem}{Theorem}[section]
\newtheorem{lemma}[theorem]{Lemma}
\newtheorem{proposition}[theorem]{Proposition}

\theoremstyle{definition}

\newtheorem{definition}[theorem]{Definition}

\def\myclearpage{}

\definecolor{darkred}{rgb}{.70,.12,.20}

\definecolor{darkgreen}{rgb}{.20,.52,.14}

\def\R{{\bf R}}
\def\N{{\bf N}}
\def\Z{{\bf Z}}
\def\C{{\bf C}}
\def\T{{\bf T}}
\def\E{{\bf E}}
\def\EO{{\cal E}_0}
\def\MO{{\cal M}_0}
\def\EH{{\cal E}_1}
\def\V{{\bf V}}

\def\words#1{\quad\hbox{#1}\quad}
\def\wwords#1{\qquad\hbox{#1}\qquad}

\numberwithin{equation}{section}

\title{
	Data Assimilation using Time-Delay Nudging\\
	in the Presence of Gaussian Noise
}
\author{Emine Celik$^1$ and Eric Olson$^2$}

\date{\today}

\begin{document}
{
    \catcode`\@=11
    \gdef\curl{\mathop{\operator@font curl}\nolimits}
}
\maketitle
\begin{center}
\textit{$^1$Department of Mathematics, 
Sakarya University\\
54050 Sakarya, Turkey}\\
\textit{$^2$Department of Mathematics and Statistics, 
University of Nevada, Reno\\
Reno, NV  89557, USA}\\
\medskip
Email addresses:  \texttt{eminecelik@sakarya.edu.tr, ejolson@unr.edu}
\end{center}

\begin{abstract}
We study a discrete-in-time data-assimilation algorithm
based on nudging through a time-delayed feedback control
in which the observational measurements have been 
contaminated by a Gaussian noise process.
In the context of the two-dimensional 
incompressible Navier--Stokes equations
we prove the expected value of the square-error
between the approximating 
solution and the reference solution over time
is proportional to the variance of the noise up to a
logarithmic correction.
The qualitative behavior and physical relevance of our analysis
is further illustrated by numerical simulation.
\end{abstract}
{\bf Keywords:} Discrete data assimilation, 
two dimensional Navier-Stokes equations.
\\

\noindent{\bf AMS subject classifications:} 35Q30, 37C50, 76B75, 93C20.


\myclearpage
\section{Introduction}\label{intro}
The goal of data assimilation 
is to combine incomplete and {\rev possibly} noisy observational measurements
taken over an interval of time with dynamical knowledge about
the system being observed to obtain a more accurate estimate of
the current state.
The time-delay nudging algorithm was introduced 
in the context
of the two-dimensional incompressible Navier--Stokes equations
by Foias, Mondaini and Titi~\cite{Foias2016} as a sequential 
data-assimilation technique to processes a discrete time series of 
incomplete noisy observations.
In this paper we extend the results of that work 
by adding a mechanism that removes outliers from
the data.  This eliminates the 
requirement that the noise be {\rev bounded}
and allows us to treat observations
contaminated by a Gaussian process.
We further remove some additional $L^2$-boundedness
assumptions on the interpolant observables
and follow~\cite{Celik2019} 
by applying a spectral filter to the feedback term.

The resulting analysis 
then provides bounds on 
$\E\big[\|U-u\|_{H^1}^2\big]$ 
where $U$ is the reference solution and $u$ is the approximating
solution obtained through the modified delay-nudging method
just mentioned.
Our main result consists of mathematically rigorous conditions under 
which bounds on the expectation naturally depend on the variance of the 
Gaussian noise process up to a logarithmic correction.

{\rev Before proceeding we remark on the importance of removing outliers
from the observations.  Foremost, this allows our analysis to treat 
Gaussian noise processes.  Intuitively with such a noise process 
there is a small chance that any
observation is contaminated by an error that dwarfs the size of all
possible trajectories lying on the global attractor. 
Therefore, upon assuming the 
unknown reference solution comes from a long term 
evolution prior to the observations, we can identify
conditions for when an observation contains an unreasonably 
large amount of noise.
Removing arbitrarily large outliers then prevents them from damaging the 
approximation obtained by data assimilation and provides rigorous bounds
on the expected quality of that approximation.

From a practical point of view, any instrumentation which provides 
measurements of a physical process has a limited numerical range of 
possible observations.  For example, the output of an anemometer 
simply can not display arbitrarily large velocities nor can it spin 
at arbitrarily large speeds.  An individual device may
include features to eliminate obviously wrong measurements, while 
signal processing techniques coupled with 
common sense are typically used 
to aggregate and reconcile data from many sources.
Thus, even if the underlying noise in a physical observation is 
Gaussian, the resulting datasets do not include 
arbitrarily large errors.  In this way the mathematical 
techniques described in this paper for removing outliers 
can also be seen to mimic realistic data collection.}

As in~\cite{Foias2016}, see also 
\cite{Azouani2014}, \cite{Bessaih2015}, \cite{Blomker2013}, 
\cite{Gesho2016} and \cite{Hayden2011} 
the model problem used for our study is the
two-dimensional incompressible Navier--Stokes equations.
In particular, the reference solution $U$ used for
our data-assimilation problem satisfies
\begin{equation}\label{ns2d}
	{\partial U\over\partial t}
		+(U\cdot\nabla)U -\nu\Delta U + \nabla p = f,
\qquad
	\nabla\cdot U = 0
\end{equation}
on the domain $\T=[0,2\pi]^2$ equipped
with $2\pi$-periodic boundary conditions.
Here $\nu$ is the kinematic viscosity, $f$ a time-independent
body force, $p$ the pressure and $U$ the Eulerian velocity
field of the fluid.

The use of the two-dimensional Navier--Stokes as a computational
tool to study data assimilation can be traced to a 1998 
report of Browning, Henshaw and Kreiss \cite{Browning1998} while
the main ideas behind the rigorous mathematical analysis trace their 
origin to a 1967 paper on determining modes by Foias and 
Prodi \cite{Foias1967}.
We note the rich analytic
theory behind these equations, their computational tractability
and the dynamical properties similar to more complicated physical 
models such as the primitive equations which govern the atmosphere.
For these reasons we hope that our current research will be seen 
as both mathematically interesting and physically relevant.

For our computations we introduce a physically-motivated interpolant 
based on observations of the velocity field 
that consist of local averages taken at a set of points in space.
From a theoretical point of view these local averages provide 
the regularity needed to obtain a type-I interpolant to which
our theory applies.
Aside from these theoretical advantages, we emphasize
from a physical point of view that such local averages approximate 
nodal measurements of the velocity field while at the same 
time more realistically represent the 
characteristic averaging properties 
of physical observational devices.
For completeness, a proof that the interpolant described above 
is type I appears in Appendix \ref{localaverages}.
{\rev Nudging with noise was investigated computationally 
with a focus on parameter recovery for the Lorenz equations
by Carlson, Hudson, Larios, Martinez, Ng and Whitehead \cite{Carlson2022} 
and in the context of the Rayleigh-B\'enard equations by
El Rahman Hammoud, Le Ma\^itre, Titi, Hoteit and 
Knio~\cite{Hammoud2023}.}

The organization of this paper is as follows.
Section \ref{preliminaries} recalls some properties of 
the two-dimensional incompressible
Navier--Stokes equations, some statistical results and 
then fixes our notation.  
We describe the process of removing outliers from 
the observational data and further state our main theoretical 
result as Theorem~\ref{mnudge} which bounds the 
expectation in the presence of Gaussian noise.
Section~\ref{bopath} obtains pathwise estimates similar to
those appearing in \cite{Foias2016} in the form 
needed for our main result which is then proved in
Section~\ref{nudgemod}.
The paper finishes in Section~\ref{compute} with a set of computations to test
the physical relevance of the theory along 
with conclusions and directions for future work.

\section{Preliminaries}\label{preliminaries}

We begin by introducing the functional notation used by
the theoretical study of the Navier--Stokes equations and then
stating some
{\it a priori\/} bounds on long-time solutions.

Let ${\cal V}$ be the set of all $\R^2$-valued divergence-free 
$2\pi$-periodic trigonometric polynomials with zero spatial
averages, $V$ the closure of ${\cal V}$ in $H^1(\T)$
where $\T=[0,2\pi]^2$ is the fundamental domain of the
$2\pi$-periodic torus, $V^*$ be the dual of $V$
and let $P_H$ be the orthogonal projection of $L^2(\T)$
onto $H$ where $H$ is the closure of ${\cal V}$ in $L^2(\T)$.
For simplicity we will write 
$L^2$ and $H^1$ without specifying the domain $\T$ 
when there is no chance of confusion.  

{\rev Due to the periodic boundary conditions, the spaces $V$ and
$H$ can also be characterized in terms of Fourier series.
In particular,
$$
    H=\bigg\{ \sum_{k\in\Z^2\setminus \{0\}} u_k e^{ik\cdot x}
        : 
        \sum_{k\in\Z^2\setminus \{0\}} |u_k|^2<\infty,\quad
        k\cdot u_k=0
        \words{and}
        u_{-k}=\overline{u_k}
    \,\bigg\}
$$
while
$$
    V=\bigg\{ \sum_{k\in\Z^2\setminus \{0\}} u_k e^{ik\cdot x}
        : 
        \sum_{k\in\Z^2\setminus \{0\}} |k|^2|u_k|^2<\infty,\quad
        k\cdot u_k=0
        \words{and}
        u_{-k}=\overline{u_k}
    \,\bigg\}.
$$
Here $u_k\in\C^2$ are the Fourier coefficients for the velocity field $u$.}

Let $A\colon V\to V^*$ and $B\colon V\times V\to V^*$
be the continuous extensions for $u,v\in{\cal V}$
of the operators given by
$$
    Au=-P_H \Delta u\qquad\hbox{and}\qquad
    B(u,v)=P_H (u\cdot\nabla v).
$$
{\rev 
Note $A$ is a positive operator with
smallest eigenvalue $\lambda_1=1$.
This dimensional constant is carried throughout our 
analysis for consistency.
}
We further write the
$L^2$ norm of $u\in H$ as $|u|$, the $H^1$ norm of $u\in V$
as $\|u\|$ and note that $|Au|$ is equivalent to the $H^2$
norm on the domain ${\cal D}(A)$ of $A$ {\rev into $H$}. 

Recall also the orthogonality property 
\begin{equation}\label{oprop}
	\big(B(v,v),Av\big)=0
\end{equation}
which holds for periodic two-dimensional divergence-free vector fields.

As shown in Constantin and Foias~\cite{Constantin1988},
Foias, Manley, Rosa and Temam~\cite{Foias2001},
Robinson~\cite{Robinson2001} or
Temam~\cite{Temam1983},
given $f\in{\rev H}$ and $U_0\in V$,
the two-dimensional incompressible Navier--Stokes equations
have a unique strong solution $U(t)\in V$ for $t\ge 0$ which depends 
continuously on the initial 
condition $U_0$ {\rev with respect to the $V$ norm}.  In particular, given
any $T>0$ we have that
\begin{equation}\label{strong}
	U\in L^\infty\big([0,T); V\big)\cap L^2\big([0,T); {\cal D}(A)\big)
\wwords{and}
	{dU\over dt}\in L^2\big([0,T);H\big).
\end{equation}
We may then express 
\eqref{ns2d} in functional form as
\begin{equation}\label{2dns}
    {dU\over dt}+\nu AU+B(U,U)=f
\end{equation}
with initial condition $U_0\in V$.

Under the conditions mentioned above, it is well known that
\eqref{2dns} possess a unique global attractor ${\cal A}$. 
To set our notation we follow \cite{Robinson2001} and
denote the bounds on ${\cal A}$ by

\begin{theorem}\label{apriori}
Let $\cal A$ be the global attractor of (\ref{ns2d}) the 
two-dimensional incompressible Navier--Stokes equations.
There exists {\em a priori} bounds $\rho_H$, $\rho_V$ and $\rho_A$ 
depending only on $\nu$ and $f$ such that
\begin{equation}
|U|\le\rho_H,\qquad
\|U\|\le \rho_V\wwords{and} |AU|\le \rho_A
\end{equation}
for every $U\in{\cal A}$.
\end{theorem}

We now characterize the linear operation $J_h$ that will be 
used to interpolate the observational measurements of the solution $U$.

\begin{definition}\label{interpolants}
A linear operator $J_h\colon V\to H$ is said 
to be a {\it type-I interpolant observable\/} if there 
exists $c_1>0$ 
such that
\beq\label{typeone}
|\Phi-J_h(\Phi)|^2\le c_1h^2\|\Phi\|^2
\wwords{for all} \Phi\in V.
\eeq
\end{definition} 
\noindent
Sometimes interpolants $I_h\colon V\to L^2(\T)$ 
{\rev which satisfy
\begin{equation}\label{interpi}
|\Phi-I_h(\Phi)|_{L^2(\T)}^2\le c_1h^2\|\Phi\|^2
\end{equation}}%
are considered.  In such cases taking $J_h=P_HI_h$ {\rev implies
$$
	|\Phi-J_h(\Phi)|^2=
	|P_H(\Phi-I_h(\Phi))|_{L^2(\T)}^2
	\le |\Phi-I_h(\Phi)|_{L^2(\T)}^2
$$ and}
results in an
interpolant which satisfies~\eqref{typeone}. Thus, no
generality is lost by
assuming the range of $J_h$ is $H$ in the first place.

Appendix~\ref{localaverages} describes the exact type-I interpolant 
used for our numerics.
{\rev In particular, Theorem \ref{thmintobs} shows the interpolant $I_h$ 
given by \eqref{noisefree} satisfies \eqref{interpi}
for every $\Phi\in V$.}
More information, other examples of type-I interpolants
as well as the definition of a type-II interpolant
may be found in \cite{Azouani2014}, \cite{Bessaih2015} and references therein.
{\rev We remark that type-II interpolants would involve the use of 
stronger Sobolev norms in the analysis and are outside the scope of the
present research.  It is an interesting question, however, whether
similar results as presented here also hold for type-II interpolants.}

To model the effects of measurement errors, we set
$$
	\widetilde J_h U(t_n) = J_h U(t_n) + \eta_n,
$$
where $\eta_n$ is sequence of independent $H$-valued Gaussian random
variables such that
$$
	\E[ \eta_n] = 0\wwords{and} \V[\eta_n]=
		\E\big[ |\eta_n|^2\big] = \sigma^2.
$$

Specifically, we suppose each $\eta_n$ is distributed as $\eta$ where
\begin{equation}\label{etadist}
	\eta=\sum_{i=1}^{2N} \sigma_i Y_i \psi_i,\qquad
	\sum_{i=1}^{2N}\sigma_i^2=\sigma^2,\qquad
	\psi_i\in H\words{with} |\psi_i|=1
\end{equation}
and $Y_i$ are independent standard normal random variables.
Note the finite degrees of freedom reflected by $2N$ is based on
the physical notion that noise in observations of 
the state of $U$ arise from a finite number of independent
noisy $\R^2$-valued measurements of the velocity field.
Thus, $N$ represents the number of measurements taken at each
instance in time and in
the two-dimensional setting considered here is inversely proportional 
to $h^2$.

Note that $|\eta|^2$ is the generalized $\chi^2$ distribution given by
$$
	|\eta|^2=\sum_{i,j}\sigma_i Y_i (\psi_i,\psi_j)\sigma_j Y_j
		=X^{T}\Psi X
\words{with} \Psi_{ij}=(\psi_i,\psi_j)
\words{and} X_i=\sigma_i Y_i.$$
In our analysis we shall make use of the exponential bound 
proved as Lemma~1 by Laurent and Massart in~\cite{Laurent2000} and 
stated here for reference as
\begin{lemma}[Laurent and Massart]\label{concentration}
Let $Y_i$ for $i=1,\ldots,d$ be independent identically-distributed standard
normal random variables and $a_i\ge 0$.  Set
$$
	|a|_\infty = \sup\big\{\, |a_i| : i=1,\ldots,d\,\big\},\quad
	|a|_2 = \Big(\sum_{i=1}^d a_i^2\Big)^{1/2}
\words{and}
	Z=\sum_{i=1}^d a_i (Y_i^2-1).
$$
Then {\rev for $x>0$ holds}
$$
	{\bf P}\big\{Z\ge 2|a|_2\sqrt x
		+ 2 |a|_\infty x\big\}\le e^{-x}.
$$
\end{lemma}

Note by setting $a_i=\sigma_i^2$ and $d=2N$ we obtain
$$
	|a|_\infty \le \sum_{i=1}^{2N} |a_i|\le \sigma^2\wwords{and}
	|a|_2 \le \Big(|a|_\infty \sum_{i=1}^{2N}|a_i|\Big)^{1/2}\le \sigma^2
$$
in which case Lemma~\ref{concentration} implies
\begin{equation}\label{concbound}
	{\bf P}\big\{
		\|X\|^2-\sigma^2 \ge 2\sigma^2\sqrt x + 2\sigma^2 x
		\big\}\le e^{-x}.
\end{equation}

With this framework in place, we describe in details 
the data assimilation method which appears in~\cite{Foias2016} that 
constitutes the beginning of
our analytical and numerical study.

\begin{definition}\label{defnudge}
The {\it delay-nudging method\/} constructs an approximation 
$u$ of the reference solution $U$
by setting $u(t_0)=0$ and then evolving $u$ continuously as
\begin{equation}\label{nudging} 
	{du\over dt}+\nu Au+B(u,u)=f +  \mu
		\big(\widetilde J_h U(t_n)-J_h u(t_n)\big) 
\words{for}
	t\in [t_n,{\rev t_{n+1}}).
\end{equation}
\end{definition}
\noindent
Here $\mu$ is a relaxation parameter which 
affects the strength of the feedback term and may 
be tuned based on the resolution and the noise present in the measurements. 

{\rev Before proceeding it is critical to check that
equations \eqref{nudging} are well posed 
and for any fixed realization of the noise process $\eta_n$ 
uniquely determine an approximating solution $u(t)$.
Since the reference solution $U(t)\in V$ for all $t\ge 0$,
then $J_h U(t_n)\in H$ for $n=0,1,\ldots$
by the definition of a type-I interpolant.
Note first that the dynamics governing $u(t)$
on the interval $[t_0,t_1)$
are identical to the two-dimensional incompressible 
Navier--Stokes equations~\eqref{2dns} with initial
condition $0\in V$ at $t=t_0$ and time-independent body forcing
$$
	f +  \mu \widetilde J_h U(t_0)=
	f +  \mu J_h U(t_0) + \eta_0\in H.
$$
Here we have used that $\eta_n\in H$.
The theory of the two-dimensional incompressible
Navier--Stokes equations now implies there exists a
unique strong solution $u\in C([t_0,t_1];V)$.
To evolve $u$ further in time
consider \eqref{nudging} with $n=1$ and initial condition 
$u(t_1)\in V$ at $t=t_1$.  The argument now follows by induction.

Suppose $u\in C([t_0,t_n];V)$.  Then $u(t_n)\in V$ implies
$J_h u(t_n)\in H$.  It follows that
$$
	f +  \mu \big(\widetilde J_h U(t_n)-J_h u(t_n)\big)=
	f +  \mu J_h U(t_n) -\mu J_h u(t_n)+ \eta_n\in H.
$$
Consequently there exists a unique strong solution 
$u\in C([t_n,t_{n+1}];V)$.  This combined with the induction
hypothesis yields that $u\in C([t_0,t_{n+1}];V)$.}

The delay-nudging method given by Definition \ref{defnudge}
was originally described and analyzed by Foias, 
Mondaini and Titi \cite{Foias2016}.  That work includes a
pathwise treatment of noisy observations for the case where
the noise process is {\rev bounded}.  Namely, one has

\begin{theorem}[Foias, Mondaini and Titi]\label{mondaini}
Let $u$ be the approximating solution obtained by 
delay-nudging on $[t_0,\infty)$ satisfying $u(t_0)=0.$
Assume $\eta_n$ is a noise process such that
\begin{equation}\label{support}
	\|\eta_n\|_{H^1} \le \EH\wwords{for all} n\in \N
\end{equation}
and that the type-I interpolant observable 
$I_h\colon L^2(\T)\to L^2(\T)$
further satisfies
\begin{equation}\label{l2hyp}
	\|I_h \Phi\|_{L^2} \le c_2 \|\Phi\|_{L^2}
\wwords{for all} \Phi\in L^2(\T).
\end{equation}
Here $c_2$ is a positive constant.
Suppose $\mu$, $h$ and $\delta$ satisfy
$$
	\mu\ge c{(\rho_V+\EH)^2\over\nu}
		\bigg( 1+\log\Big({\rho_V+\EH\over\nu\lambda_1^{1/2}}\Big)\bigg),
\qquad
	h\le {1\over 2c_0} \Big({\nu\over\mu}\Big)^{1/2}
$$
and 
$$
	\delta\le {c\over\mu}\min\Big\{
	1, {\nu^{3/2}\mu^{1/2}\over \rho_H\rho_V},
	{\nu^2\lambda_1^{1/2}\over \rho_H\rho_V},
	{\nu^2\lambda_1\over (\rho_V+\EH)^2},
	{(\nu\lambda_1)^{1/2}\over\mu^{1/2}},
	{(\nu\lambda_1)^2\over\mu^2}\Big\}.
$$
Then
$$
	\limsup_{t\to\infty} \|U-u\|_{H^1} \le c\EH.
$$
\end{theorem}
This paper extends the above result to the 
case when $\eta_n$ is an $H$-valued Gaussian noise processes that 
{\rev is not bounded by} $\EH$. 
We further remove the {\rev continuity} condition~\eqref{l2hyp} and
the requirement that $\eta_n\in V$. 

To these ends we consider a
modified nudging method that removes outliers along with a 
bootstrapping argument to obtain a 
theorem that applies when the noise process is Gaussian.
In particular, by defining the modified 
interpolant observable 
\begin{plain}\begin{equation}\label{I0def}
    \widetilde J_h^o U(t_n) = \cases{
        \widetilde J_h U(t_n)&
            for $|\widetilde J_h U(t_n)| \le 2M$\cr
        0 & otherwise
    }
\end{equation}\end{plain}%
we obtain a new noise process
$\eta^o_n=\widetilde J_h^o U(t_n)-J_h U(t_n){\rev{}\in H}$ 
which {\rev is bounded}
and an interpolant that filters 
outliers that correspond to points outside the known 
absorbing ball of the global attractor
when $M$ is large enough.
Replacing $\widetilde J_h$ by $\widetilde J_h^o$ in
\eqref{nudging}
subsequently leads to the {\it modified nudging method\/} and 
our main theoretical result.

\begin{theorem}\label{mnudge}
Let $u$ be an approximating solution obtained by the modified
nudging method with suitable relaxation parameter $\mu$.
If $\delta>0$ and $h>0$ are small enough, then 
there exists a constant $C_0$ independent of $\sigma$ 
and a logarithmic correction $f(\sigma)$ such that
$$
	\limsup_{t\to\infty} \E\big[\|U-u\|^2\big]
		\le C_0\sigma^2 f(\sigma)
$$
holds for all $\sigma$ sufficiently small.
\end{theorem}

We remark that Theorem \ref{mnudge} forgoes pathwise bounds in order to
handle the case where the measurement errors are 
distributed according to a Gaussian distribution.
Since in the Gaussian case there is no finite $\EO$ such that 
$|\eta_n|\le \EO$ holds for all $n$
no matter how small the variance $\sigma^2$, with
non-zero probability
there will be arbitrarily long sequences
of consecutive observations such that $|\eta_n|$ is large.
It follows there exists $C>0$ such that
the approximating solutions obtained by
the original nudging method satisfy
$$
	\limsup_{t\to\infty}|U-u|\ge C
$$
almost surely no matter how small $\sigma$.
For this reason Theorem~\ref{mnudge} does
not provide pathwise bounds, but instead bounds the expectation in a way
that shows the corresponding approximating solutions asymptotically recover
the exact solution when $\sigma$ goes to zero. 

We now recall some inequalities that will be used in the
subsequent analysis.
Writing the smallest eigenvalue of the Stokes operator $A$ as
$\lambda_1=(2\pi/L)^2$ we have the
Poincar\'e inequalities
\begin{equation}\label{poincarreV}
    \lambda_1 |U|^2\le \|U\|^2\wwords{for} U\in V
\end{equation}
and
\begin{equation}\label{poincarreDA}
    \lambda_1^2 |U|^2\le \lambda_1 \|U\|^2\le |AU|^2 
\wwords{for}
	U\in\mathcal{D}(A).
\end{equation}

Recall also the combination of Agmon's 
inequality \cite{Agmon2010}
with \eqref{poincarreDA} given by
\begin{equation}\label{agmon}
\|U\|_{L^\infty}\le C|U|^{1/2}|AU|^{1/2}
	\le  C\lambda_1^{-1/2}|AU|
\end{equation}
and Ladyzhenskaya's inequality
\begin{equation}\label{lady}
	\|U\|_{L^4} \le C |U|^{1/2} \|U\|^{1/2}.
\end{equation}
In both \eqref{agmon} and \eqref{lady} the constant $C$ is dimensionless and
does not depend on $U$.

Applying H\"older followed by the Ladyzhenskaya and
Agmon inequalities to the nonlinear terms in the 
two-dimensional Navier--Stokes equations yields
\begin{equation}\label{bbone}
|B(u,v)|\le \|u\|_{L^4} \|A^{1/2}v\|_{L^4}
       \le C |u|^{1/2} \|u\|^{1/2} \|v\|^{1/2} |Av|^{1/2}
\end{equation}
and
\begin{equation}\label{bbtwo}
|B(u,v)|\le \|u\|_{L^\infty} \|A^{1/2}v\|_{L^2}
   \le C |u|^{1/2} |Au|^{1/2} \|v\|.
\end{equation}

Finally, we shall make use of the logarithmic bound on the non-linear 
term proved in 
Titi~\cite{Titi1987}, further employed in \cite{Foias2016} and stated
here as
\begin{proposition}\label{etiti} If $U$ and $w$ are in ${\cal D}(A)$ then 
$$
	\big|\big(B(w,U)+B(U,w),Aw\big)\big| 
	\le C \|w\| \|U\|\Big( 
	1+\log{|AU|\over \lambda_1^{1/2}\|U\|}\Big)^{1/2} |Aw|,
$$
where $C$ is a non-dimensional constant depending only on the domain.
\end{proposition}

\section{Bounds on the Paths}\label{bopath}

In this section we adapt the proof of Theorem~\ref{mondaini} that
appears in \cite{Foias2016} to obtain a similar pathwise bound
under the weaker
hypothesis that the $L^2$ norm of $\eta_n$ is bounded and without 
the requirement
that $J_h$ also satisfy \eqref{l2hyp}.  Unlike the original proof
we limit our attention to {\rev periodic} domains $\T$ which later 
form the context of our computational setting.
Since the plan is to use this result to obtain bounds on
the expectation in the next section, we explicitly track
the rate of approximate synchronization over time.

Write $w=U-u$ 
where $U$ is a free running solution to the two-dimensional
Navier--Stokes equations given by \eqref{2dns} and $u$ is an 
approximating solution obtained from the delay nudging 
method \eqref{nudging}.
{\rev It has already been shown that $U$ and $u$ are strong solutions of
the two-dimensional incompressible Navier--Stokes equations
on each interval $[t_n,t_{n+1})$,
though each with a different initial condition and body force.
Consequently $w$ enjoys the same regularity properties.}
It follows that 
\begin{equation}\label{dnmethod}
	{dw\over dt} +\nu Aw+B(w,U)+B(U,w){\rev{}+{}}B(w,w)=
	-\mu J_h w(t_n) -\mu \eta_n
\end{equation}
for $t\in [t_n,{\rev t_{n+1}})$.
To treat the time delay resulting from the first term on the
right in the above equation we first prove
\begin{lemma}\label{dodelay}
Assuming $|\eta_n|\le \alpha$ then
for $t\in [t_n,t_{n+1})$ holds
\begin{align*}
	|w(t)-w(t_n)|^2 & \le 
		4(t-t_n)^2\mu^2
	\big\{(c_1h^2+\lambda_1^{-1}) 	
			\|w(t_n)\|^2 + \alpha^2\big\} \\
		& \qquad+
		4(t-t_n)\int_{t_n}^{t}
		\big\{\nu + C\lambda_1^{-1/2}\big(2\rho_V+\|w{\rev (s)}\|\big)\big\}^2
		|Aw{\rev (s)}|^2{\rev ds}.
\end{align*}
\end{lemma}  
\begin{proof}
Since the Cauchy--Schwarz inequality implies
$$
   |w(t)-w(t_n)|^2\le \Big(\int_{t_n}^t \Big|{dw(s)\over ds}\Big| ds\Big)^2
   \le (t-t_n) \int_{t_n}^t \Big|{dw(s)\over ds}\Big|^2 ds,
$$
it is enough to estimate
$$
	\Big|{dw\over dt}\Big| \le
		\nu |Aw| + |B(w,U)|+|B(U,w)|+|B(w,w)|
		+ \mu |J_h w(t_n)| + \mu |\eta_n|.
$$
Applying \eqref{bbone} and \eqref{bbtwo} 
to the nonlinear terms followed by the
Poincar\'e inequality yields
$$
|B(w,U)| 
	\le C |w|^{1/2} |Aw|^{1/2} \|U\|
	\le C \lambda_1^{-1/2} \rho_V |Aw|
$$
$$
|B(U,w)|
	\le C |U|^{1/2} \|U\|^{1/2} \|w\|^{1/2} |Aw|^{1/2}
	\le C \lambda_1^{-1/2} \rho_V |Aw|
$$
$$
|B(w,w)| 
	\le C |w|^{1/2} \|w\| |Aw|^{1/2}
	\le C \lambda_1^{-1/2} \|w\| |Aw|.
$$
Note also
$$
	|J_h w(t_n)| \le |w(t_n)-J_h w(t_n)| + |w(t_n)|
		\le (c_1^{1/2} h +\lambda_1^{-1/2})\|w(t_n)\|.
$$
Therefore
$$
	\Big|{dw\over dt}\Big|\le
	\big\{\nu + C\lambda_1^{-1/2}\big(2\rho_V+\|w\|\big)\big\}|Aw|
		+\mu (c_1^{1/2} h +\lambda_1^{-1/2})\|w(t_n)\|+\mu \alpha
$$
and so
$$
	\Big|{dw\over dt}\Big|^2\le
	4\big\{\nu + C\lambda_1^{-1/2}\big(2\rho_V+\|w\|\big)\big\}^2|Aw|^2
		+4\mu^2 (c_1 h^2 +\lambda_1^{-1})\|w(t_n)\|^2+4\mu^2 \alpha^2.
$$
Consequently,
\begin{align*}
	|w(t)-w(t_n)|^2 & \le 
		4(t-t_n)^2\mu^2
	\big\{(c_1h^2+\lambda_1^{-1}) 	
			\|w(t_n)\|^2 + \alpha^2\big\} \\
		& \qquad +
		4(t-t_n)\int_{t_n}^{t}
		\big\{\nu + C\lambda_1^{-1/2}\big(2\rho_V+\|w{\rev (s)}\|\big)\big\}^2
		|Aw{\rev (s)}|^2 {\rev ds},
\end{align*}
which was to be shown.
\end{proof}

With the above lemma in hand we now turn our attention to a version
of Theorem~\ref{mondaini} that will be used to prove 
our main theoretical result.
In addition to the already mentioned differences, 
Proposition~\ref{improvedpath} below provides an explicit
separation of the dependency of the exponential rate $\theta$ on
the maximum bound $\EO$ from the estimate on $\|w(t)\|$ which results
when $|\eta_n|$ is further bounded for a finite time by $\alpha$.
Note this separation of dependencies is required for the probabilistic
estimates which appear in Section~\ref{nudgemod}.

\begin{proposition}\label{improvedpath}
Let $\MO$ and $\EO$ be fixed.  There are positive
constants $c$, $\delta$, $\theta$, $h$ and $\mu$ with $\theta<1$
such that upon taking $t_n=t_0+n\delta$ the conditions
$\|w(t_k)\|\le\MO$ 
and $|\eta_n|\le \alpha$ where
$\alpha\le\EO$ for $n=k,\ldots,k+p$ imply
\begin{equation}\label{mygexp}
	\|w(t)\|^2\le \theta^{n-k}
	\|w(t_k)\|^2+
		c\alpha^2
\words{for all} t\in[t_n,t_{n+1}]
\end{equation}
and $n=k,\ldots,k+p$.
\end{proposition}
\begin{proof}
Taking the inner product of \eqref{dnmethod} with $Aw$ and using the
orthogonality~\eqref{oprop} we have
\begin{align}\label{wevolve}
	{1\over 2}{d\|w\|^2\over dt}+\nu|Aw|^2
		&+\big(B(w,U)+B(U,w),Aw\big)
	= -\mu\big(J_h w(t_n)+\eta_n,Aw\big).
\end{align}
First, use Proposition~\ref{etiti} to estimate the non-linear terms as
$$
	\big|\big(B(w,U)+B(U,w),Aw\big)\big| 
	\le L \|w\| |Aw| \le {2L^2\over \nu} \|w\|^2 + {\nu\over 8}|Aw|^2,
$$
where
$$
	L=
	C\rho_V\Big( 1+\log{\rho_A\over \lambda_1^{1/2}\rho_V}\Big)^{1/2}.
$$
The last term in the inner product on the right side of \eqref{wevolve} 
may be estimated as
$$
	\mu\big|(\eta_n,Aw)\big|
	\le \mu|\eta_n||Aw|
	\le \mu \alpha|Aw|
	\le 
{2\mu^2\over\nu} \alpha^2
	+{\nu\over 8} |Aw|^2.
$$

Now write
$$
	-\mu\big(J_h w(t_n),Aw\big)
	=
	\mu\big(w(t_n)-J_h w(t_n),Aw\big)
	+\mu\big(w-w(t_n),Aw\big)
	-\mu \|w\|^2.
$$
Since the Cauchy-Schwarz inequality and \eqref{typeone} implies
\begin{align*}
	\mu\big|\big(w(t_n)-J_h w(t_n),Aw\big)\big|
	&\le {2\mu^2\over\nu}c_1 h^2\|w(t_n)\|^2
		+{\nu\over 8}|Aw|^2
\end{align*}
and
\begin{align*}
	\mu\big|\big(w-w(t_n),Aw\big)\big|
	&\le
	{2\mu^2\over\nu}|w-w(t_n)|^2 +{\nu\over 8}|Aw|^2,
\end{align*}
it follows that
\begin{align*}
	{d\|w\|^2\over dt}+2\Big(
		\mu-{2L^2\over \nu}
		\Big)\|w\|^2&+\nu|Aw|^2 \le
	{4\mu^2\over\nu}
	\big(
		|w-w(t_n)|^2+c_1 h^2\|w(t_n)\|^2+\alpha^2
	\big).
\end{align*}
Choosing $\mu$ large enough that $\mu\nu\ge 4L^2$ and 
applying Lemma~\ref{dodelay} we have
\begin{align*}
	{d\|w\|^2\over dt}+\mu \|w\|^2+\nu|Aw|^2 &-
		{16\mu^2(t-t_n)\over \nu}\int_{t_n}^{t}
		\big\{\nu + C\lambda_1^{-1/2}\big(2\rho_V+\|w{\rev (s)}\|\big)\big\}^2
		|Aw{\rev (s)}|^2 {\rev ds}\\
	&\le R_1(t-t_n) \|w(t_n)\|^2+ R_2(t-t_n)\alpha^2,
\end{align*}
where
$$
	R_1(\tau)= {4\mu^2\over\nu}
		\big\{4\tau^2\mu^2(c_1h^2+\lambda_1^{-1})+ c_1h^2
		\big\}
\wwords{and}
	R_2(\tau)= {4\mu^2\over\nu}
		(4\tau^2\mu^2+1).
$$
Consequently
\begin{align*}
	{d\over dt}\big(\|w\|^2e^{\mu t}\big)&+e^{\mu t_n}\nu|Aw|^2 \\
		&-
		e^{\mu t_{n+1}}{16\mu^2(t-t_n)\over \nu}\int_{t_n}^{t}
		\big\{\nu + C\lambda_1^{-1/2}\big(2\rho_V+\|w{\rev (s)}\|\big)
			\big\}^2|Aw{\rev (s)}|^2 {\rev ds}\\
	&\le e^{\mu t}R_1(t-t_n)\|w(t_n)\|^2 + e^{\mu t}R_2(t-t_n)\alpha^2.
\end{align*}

Integrate over the interval $[t_n,t]$.  Estimate the resulting 
double integral as
\begin{align*}
		\int_{t_n}^t
		{16\mu^2(\tau-t_n)\over \nu}\int_{t_n}^{\tau}&
		\big\{\nu + C\lambda_1^{-1/2}\big(2\rho_V
			+\|w(s)\|\big)\big\}^2|Aw(s)|^2 ds\, d\tau\\
	&\le
		{16\mu^2(t-t_n)^2\over \nu}\int_{t_n}^{t}
		\big\{\nu + C\lambda_1^{-1/2}\big(2\rho_V
			+\|w{\rev (s)}\|\big)\big\}^2|Aw{\rev (s)}|^2 {\rev ds}.
\end{align*}
Then use the fact that $R_1$ and $R_2$ are increasing functions to obtain
\begin{align*}
	\|w\|^2e^{\mu t}&+
		e^{\mu t_n} \int_{t_n}^t \Big(\nu
			-e^{\mu \delta}{16\mu^2(t-t_n)^2\over \nu}
			\big\{\nu + C\lambda_1^{-1/2}\big(2\rho_V+\|w{\rev (s)}\|\big)\big\}^2
		\Big)|Aw{\rev (s)}|^2 {\rev ds}\\
	&\le 
\|w(t_n)\|^2e^{\mu t_n}+
	{e^{\mu t}-e^{\mu t_n}\over\mu}\Big\{R_1(t-t_n) \|w(t_n)\|^2
		+R_2(t-t_n)\alpha^2\Big\}.
\end{align*}
Therefore,
\begin{align}\label{wrecur}
	\|w\|^2&+
		e^{-\mu (t-t_n)} \int_{t_n}^t \Big(
			\nu -e^{\mu \delta}{16\mu^2(t-t_n)^2\over \nu}
			\big\{\nu + C\lambda_1^{-1/2}\big(2\rho_V
		+\|w{\rev (s)}\|\big)\big\}^2
		\Big)|Aw{\rev (s)}|^2{\rev ds}\nonumber\\
	&\le 
	\varphi(t-t_n)
\|w(t_n)\|^2
		+{1-e^{-\mu(t-t_n)}\over\mu}R_2(t-t_n) \alpha^2,
\end{align}
where $$\varphi(\tau)=e^{-\mu \tau}+{1-e^{-\mu\tau}\over\mu}R_1(\tau).$$

Since
$$
	\varphi'(\tau)=-\mu e^{-\mu t}+e^{-\mu\tau}
		R_1(\tau)+{1-e^{-\mu\tau}\over\mu}R_1'(\tau)
$$
choosing $h$ small enough implies
$$
	\varphi'(0)
	=-\mu+R_1(0)
	=-\mu+{4\mu^2 \over\nu} c_1h^2<0.
$$
As $\varphi(0)=1$ and $R_1(0)<\mu$ it follows
there is $\varepsilon>0$ such that 
$$\varphi(\tau)< 1
\wwords{and} R_1(\tau)<\mu
\wwords{for}
\tau\in(0,\varepsilon].$$
We emphasize at this point that $\varepsilon$ is independent
of $\alpha$, $\MO$ and $\EO$ but depends on $h$, $\mu$, $\nu$ and
other parameters of the system.

Now set
$$
	S^2= \max\Big\{\, \MO^2,\, 
		{1-e^{-\mu\varepsilon}\over\mu} R_2(\varepsilon)\EO^2,\,
		{R_2(\varepsilon)\over\mu-R_1(\varepsilon)}\EO^2
	\,\Big\}
$$ 
and choose $\delta\le\varepsilon$ so small that
$$
			\nu -e^{\mu \delta}{16\mu^2\delta^2\over \nu}
			\big\{\nu + C\lambda_1^{-1/2}\big(2\rho_V+2^{1/2}S\big)\big\}^2
		\ge 0.
$$
By hypothesis $\|w(t_k)\|\le \MO\le S<2^{1/2}S$.
Under the assumption on $\delta$ we obtain
$$
            \nu -e^{\mu \delta}{16\mu^2(t-t_n)^2\over \nu}
            \big\{\nu + C\lambda_1^{-1/2}\big(2\rho_V+\|w{\rev (s)}\|\big)\big\}^2> 0
$$
{\rev for $n=k$ at $s=t_k$ and any $t\in[t_k,t_{k+1}]$.}
Therefore, by continuity
the integral term on the left-hand side of \eqref{wrecur} drops out
over some non-trivial maximal interval.  But then
\begin{align*}
	\|w(t)\|^2
	&\le \varphi(t-t_k) \|w(t_k)\|^2
		+{1-e^{-\mu(t-t_k)}\over \mu}R_2(t-t_k)\alpha^2\\
	&\le \|w(t_k)\|^2
		+{1-e^{-\mu\varepsilon}\over\mu}R_2(\varepsilon)\EO^2
		\le 2S^2
\end{align*}
shows this interval is at least as large as $[t_{\rev k},t_{\rev k+1}]$.

Taking $t=t_{k+1}$ and $\theta=\varphi(\delta)$ 
yields
$$
	    \|w(t_{k+1})\|^2
    \le \theta \|w(t_k)\|^2 +{1-e^{-\mu\delta}\over\mu}R_2(\delta)\alpha^2.
$$
Moreover, since
$$
	\theta=e^{-\mu\delta}+{1-e^{-\mu\delta}\over\mu}R_1(\delta)
\wwords{implies}
	1-\theta={1-e^{-\mu\delta}\over\mu}\big(\mu- R_1(\delta)\big)
$$
then substituting further obtains
\begin{align*}
	\|w(t_{k+1})\|^2
	&\le \theta\|w(t_k)\|^2+ (1-\theta) {R_2(\delta)\over
		\mu-R_1(\delta)}\alpha^2.
\end{align*}
Noting $R_2(\tau)/\big(\mu-R_1(\tau)\big)$ is an increasing function of
$\tau$ and that $\alpha\le\EO$ results in
$$
	\|w(t_{k+1})\|^2\le \theta\|w(t_k)\|^2+ (1-\theta) {R_2(\varepsilon)\over
		\mu-R_1(\varepsilon)}\EO^2\le S^2.
$$

By induction it follows that $\|w(t_n)\|^2\le S^2$ 
holds for all $n=k,\ldots,k+p$ and consequently
$$
	\|w(t_{n+1})\|^2
	\le \theta\|w(t_n)\|^2+ (1-\theta) 
		{R_2(\delta)\over
		\mu-R_1(\delta)}\alpha^2
\words{for all} n=k,\ldots,k+p.
$$
Another induction immediately yields that
\begin{align*}
	    \|w(t_{n})\|^2
    &\le \theta^{n-k} \|w(t_k)\|^2
        + (1-\theta) {R_2(\delta)\over
        \mu-R_1(\delta)} \alpha^2 \sum_{j=0}^{n-k-1} \theta^j\\
    &\le \theta^{n-k} \|w(t_k)\|^2
        + {R_2(\delta)\over
        \mu-R_1(\delta)} \alpha^2
\end{align*}
and also 
for $t\in[t_n,t_{n+1}]$ that
$$
	\|w(t)\|^2\le \varphi(t-t_n) \|w(t_n)\|^2
		+ {1-e^{-\mu(t-t_n)}\over\mu} R_2(t-t_n)\alpha^2.
$$
Finally, combining the above two inequalities we obtain
$$
	\|w(t)\|^2\le \|w(t_n)\|^2+
        {R_2(\delta)\over
        \mu-R_1(\delta)} \alpha^2
    \le \theta^{n-k} \|w(t_k)\|^2
        + c\alpha^2,
$$
where $c=2R_2\big(\delta)/(\mu-R_1(\delta)\big)$.
Noting that $c$, $\delta$, $\theta$, $h$ and $\mu$ depend on $\EO$ and 
$\MO$ but are independent of $\alpha$ and $\|w(t_k)\|$ finishes the proof.
\end{proof}

\section{Bounds on the Expectation}\label{nudgemod}

In this section we employ the modification \eqref{I0def} to the 
delay-nudging method of \cite{Foias2016} to handle the 
outliers that result from Gaussian noise processes.
Our criterion for the identification of outliers in the observational
data is based on known {\it a priori\/} bounds of the global attractor 
for the reference solution.

Intuitively, if $M$ is very large then $\widetilde J_h^o$ is the same 
as $\widetilde J_h$ most of the time.  On the other hand, the few cases
where these interpolants are different is sufficient to ensure the 
probability distribution behind the modified noise process $\eta_n^o$ 
satisfies the hypothesis of Proposition~\ref{improvedpath} and 
{\rev is bounded}.
We remark that the $\eta_n^o$ do not form a 
sequence of independent identically distributed random variables since
the removal of outliers is relative to the size of $U(t_n)$ 
at different points of time.

Fortunately, the pathwise bounds obtained 
in the previous section
do not require the noise be independent or identically distributed;
therefore, upon noting
$\widetilde J_h^o U(t_n)=J_h U(t_n)+\eta_n^o$ and 
$|\eta_n^o|\le\EO$ for suitable choice of $M$ and $\EO$,
Proposition~\ref{improvedpath} can be applied directly to the 
modified nudging method.  To find a suitable $M$ and $\EO$ 
we first prove

\begin{lemma}\label{noisebnd}
If $M \ge \rho_H+2\pi c_1^{1/2} \rho_V$, then
for all $h\le2\pi$ holds~$|\eta^o_n|\le \EO$ where $\EO=3M$.
\end{lemma}
\begin{proof}
First note that
\begin{equation}\label{onem}
		|J_h U(t_n)|
		\le |U(t_n)|+|U(t_n)-J_h U(t_n)|
		\le \rho_H+c_1^{1/2}h\|U(t_n)\|
		\le M.
\end{equation}
Now, if $|\widetilde J_h U(t_n)|>2M$ then $\widetilde J_h^o U(t_n)=0$.
Therefore
\begin{align*}
	|\eta^o_n|
		&=|\widetilde J_h^o U(t_n)-J_h U(t_n)|
		=|J_h U(t_n)|\le M.
\end{align*}
If $|\widetilde J_h U(t_n)|\le 2M$ then $\eta^o_n=\eta_n$.
Therefore
\begin{align*}
	|\eta^o_n|
		&
	=|\eta_n|
		=|\widetilde J_h U(t_n)-J_h U(t_n)|
		\le
		|\widetilde J_h U(t_n)|+ |J_h U(t_n)|\le 2M+M=3M.
\end{align*}
Since $|\eta^o_n|\le \EO$ in either case, the
bound follows.  
\end{proof}

We remark that since $|\eta^o_n|\le\EO$ for all $n$
it immediately follows from Proposition~\ref{improvedpath} that there
exists $c$, $\delta$, $h$ and $\mu$
such that
\begin{equation}\label{badbound}
	\limsup_{t\to\infty} \|U-u\|^2 \le c \EO^2.
\end{equation}
The drawback of \eqref{badbound} is the upper bound does 
not depend on the variance of the Gaussian noise process.
Moreover, as there is always a chance of obtaining a long 
sequence of outliers in a row, it appears not possible to 
obtain better pathwise bounds.

{\rev For example, if the observational errors were such 
that $|\widetilde J_h U(t_n)|>2M$ for all $t_n$ over an
interval of length $T$, then for those $t_n$ the modified
nudging method would evolve the approximating 
solution independently of the observations as
$$
	{du\over dt}=\nu Au+B(u,u) = f - J_h u(t_n)
\words{for} t\in[t_n,t_{n+1}).
$$
Although these dynamics are different than those which 
govern the reference solution, the main difficulty
is lack of any coupling along with sensitive dependence
on initial conditions would lead $U$ and $u$ to 
become decorrelated when $T$ is large.
Thus,
no pathwise bounds qualitatively better than \eqref{badbound} 
are available because a sequence 
of such outliers, though unlikely, will repeatedly 
occur with probability one over any infinite period of time.}

We now proceed to the proof of 
Theorem~\ref{mnudge},
our main theoretical result,
which provides a bound on
$\E\big[\|U-u\|^2\big]$ that vanishes
as $\sigma^2$ tends to zero.

\begin{proof}[Proof of Theorem~\ref{mnudge}]
Fix $\MO=\rho_V$ and $\EO=3M$ 
where $M=\rho_H+2\pi c_1^{1/2}\rho_V$.
Note that 
$u(t_0)=0$ and $U(t_0)\in{\cal A}$ implies
$$\|w(t_0)\|=\|U(t_0)-u(t_0)\|=\|U(t_0)\|\le\rho_V=\MO.$$
Now, provided $h\le 2\pi$ then Lemma~\ref{noisebnd} implies 
the noise present in
the modified interpolant satisfies
$\|\eta^o_n\|_{L^2}\le \EO$ for all $n$.
It follows from Proposition~\ref{improvedpath}
that there are positive constants $c$, $\delta$,
$\theta$, $h$ and $\mu$ such that~\eqref{mygexp} holds.
Therefore, upon setting $\alpha=\EO$ we obtain
$$
	\|w(t)\|^2 \le cK
\words{for all} t\ge t_0
\wwords{where}
cK=\rho_V^2+c \EO^2.
$$

We remark that the choice of $h$ provided above fixes the 
interpolant $J_h$
and consequently the exact form of the noise process $\eta_n$ in 
$\widetilde J_h$.
In particular $N$ and the unit vectors $\psi_i\in H$ for 
$i=1,\ldots,2N$ should be considered fixed for the rest of the proof.

For each $\alpha$
such that $0<\alpha\le M$ let $n_\alpha$ be the value of $n$
such that
$ \theta^n K \le \alpha^2 < \theta^{n-1} K$.
Since $K>\EO^2> \alpha^2$ it follows that $n_\alpha\ge 1$.
Consequently, if 
$$|\eta_n|\le\alpha
\wwords{for} n=k,\ldots,k+n_\alpha$$
then by $\eqref{onem}$ we have
$$
	|\tilde J_h U(t_n)|=|J_h U(t_n)|+|\eta_n| \le M+\alpha \le 2M.
$$
This means $\eta_n=\eta^o_n$ and so $|\eta^o_n|\le\alpha$.
It again follows from \eqref{mygexp} with the same choices of
$c$, $\delta$, $\theta$, $h$ and $\mu$ as before that
$$
	\|w(t)\|^2\le \theta^{\rev n-k} \|w(t_k)\|^2+c\alpha^2
		\le 2c\alpha^2
\words{for} t\in [t_n,t_{n+1}] \words{where} n={\rev k+n_\alpha}.
$$
Therefore,
\begin{align}\label{pest1}
    {\bf P}\big\{\|w(t)\|^2 \le 2c\alpha^2\big\}
		&\ge
	{\bf P}\big\{|\eta_n| \le \alpha
		\hbox{ for } n=k,\ldots,k+n_\alpha\big\}\nonumber\\
	&= {\textstyle \prod_{n=k}^{k+n_\alpha} }\,
		{\bf P}\big\{|\eta_n| \le \alpha\big\}.
\end{align}

Recall the $\eta_n$ are distributed as in \eqref{etadist} where
$|\eta|^2=X^T \Psi X \le \|\Psi\| \|X\|^2$.
Here $\|\Psi\|$ denotes the spectral norm of the $2N\times 2N$ matrix $\Psi$
and $\|X\|$ the Euclidean norm of the random vector~$X$.
Thus,
\begin{equation}\label{Xineq}
	{\bf P}\big\{|\eta_n|\le \alpha\big\}=
	{\bf P}\big\{|\eta|\le \alpha\big\}
		\ge {\bf P}\big\{\|X\|^2 \le \alpha^2/\|\Psi\|\big\}.
\end{equation}
We emphasize that $\|\Psi\|$ is simply a constant at this point
as we have already fixed~$N$ and~$\psi_i$ earlier 
in the proof.  Note also that $\Psi$ is
independent of $\alpha$ and $\sigma$.

Substituting \eqref{Xineq} into \eqref{pest1} yields
$$
	{\bf P}\big\{\|w(t)\|^2 \le 2c\alpha^2\big\}\ge
		\varphi
\wwords{where}
		\varphi={\bf P}\big\{\|X\|^2 \le 
			\alpha^2/\|\Psi\|\big\}^{{\rev n_\alpha}+1}.
$$
Since 
$$
	\E\big[\|w(t)\|^2\big]
		\le 2c\alpha^2{\bf P}\big\{\|w(t)\|^2\le 2c\alpha^2\big\}
		+{\bf P}\big\{\|w(t)\|^2>2c\alpha^2\big\} cK,
$$
then $2c\alpha^2\le 2c M^2=2c(\EO/3)^2< cK$ implies
\begin{equation}\label{modthis}
	\E\big[\|w(t)\|^2\big]
		\le \varphi(2c\alpha^2) + (1-\varphi)cK
\words{for}
	t\in[t_n,t_{n+1}]\words{where} n=n_\alpha.
\end{equation}

Now rewrite the bound in \eqref{modthis}
in terms of $\sigma$ by choosing a suitable 
$\alpha$.  By definition
$$
	n_\alpha \log\theta \le \log(\alpha^2/K) < (n_\alpha-1)\log\theta
$$
and so
$$
	n_\alpha-1 \le {\log(K/\alpha^2)\over\log(1/\theta)} < n_\alpha.
$$
Therefore,
$$
	\varphi \ge
{\bf P}\big\{ \|X\|^2 \le \alpha^2/\|\Psi\|\big\}
	^{\log(K/\alpha^2)/\log(1/\theta)+2}
=
{\bf P}\big\{ \|X\|^2 \le \alpha^2/\|\Psi\|\big\}
	^{(\log(K/\alpha^2)+\kappa)/\log(1/\theta)},
$$
where $\kappa=2\log(1/\theta)$.

Assuming $8\sigma^2\le\alpha^2/\|\Psi\|$, then 
{\rev Lemma}~\ref{concentration} used as \eqref{concbound} implies
\begin{align*}
	{\bf P}\big\{\|X\|^2\le \alpha^2/\|\Psi\|\big\}
		= 1-{\bf P}\big\{\|X\|^2\ge \alpha^2/\|\Psi\|\big\}
		\ge 1- e^{-x},
\end{align*}
where {\rev $x>0$ is defined by}
$2\sigma^2\sqrt x + 2\sigma^2x = \alpha^2/\|\Psi\|-\sigma^2$.
Simple estimates then yield
$$
	x\ge 
		\Big({\sqrt{15}- 1\over \sqrt{32}}\Big)^2
	{\alpha^2\over \sigma^2\|\Psi\|}
	> 
		{\alpha^2\over 4\sigma^2\|\Psi\|}.
$$
{\rev We remark that although $\sigma^2<\alpha^2/\|\Psi\|$ is
sufficient to apply Lemma~\ref{concentration}, the explicit
lower bound on $x$ scaling as $\alpha^2/\sigma^2$ is used
in our subsequent estimates.}
It follows that
$$
	{\bf P}\{\, \|X\|^2\le \alpha^2/\|\Psi\|\,\} 
	\ge 1- e^{-\gamma \alpha^2/\sigma^2}
\wwords{where}
	\gamma={1\over 4\|\Psi\|}.
$$

By Bernoulli's inequality
$$
	\varphi
	\ge \big(1-e^{-\gamma\alpha^2/\sigma^2}\big)^{(\log(K/\alpha^2)
		+\kappa)/\log(1/\theta)}
	\ge 1-e^{-\gamma\alpha^2/\sigma^2}\big(\log(K/\alpha^2)
		+\kappa\big)/\log(1/\theta)
$$
so that
$$
	1-\varphi
	\le e^{-\gamma\alpha^2/\sigma^2}
		\big(\log(K/\alpha^2)+\kappa\big)/\log(1/\theta).
$$
{\rev Since ${\cal E}_0=3M$ and $\alpha\le M$
then $9\alpha^2\le{\cal E}_0^2$.  The definition $cK=\rho_V^2+c{\cal E}_0^2$
then implies $9\alpha^2<K$ or that $K/\alpha^2\ge 9$.
Moreover, $\kappa/\log(1/\theta)=2$.  Therefore}
$$
	\log(K/\alpha^2)+\log\Big(
        {\log(K/\alpha^2)+\kappa\over\log(1/\theta)}\Big)\ge \log 9+\log 2
		=\log(18).
$$
Setting
\begin{equation}\label{sigmaval}
		\sigma^2 = \gamma
	\alpha^2\Big\{
		\log(K/\alpha^2)+\log\Big(
		{\log(K/\alpha^2)+\kappa\over\log(1/\theta)}\Big)\Big\}^{-1}
\end{equation}
implies
$$
	8\sigma ^2 \le {8\gamma \alpha^2\over\log(18)}{\rev{}={}}
		{8\alpha^2 \over 4 \|\Psi\| \log(18)} <
		{\alpha^2\over \|\Psi\|}.
$$
Thus, the previous assumption on the smallness of $\sigma$ is satisfied.
Moreover, it follows that
$$
   (1-\varphi)K \le \alpha^2.
$$
Consequently
$$
	\E\big[ \|w(t)\|^2\big] \le
	\varphi (2c\alpha^2)+(1-\varphi)cK \le 3c \alpha^2.
$$

Finally, since $8\sigma^2\le \alpha^2/\|\Psi\|$ then
$$
	\log(K/\alpha^2)\le \log(C_1/\sigma^2)
	\wwords{where} C_1={K\over 8\|\Psi\|}.
$$
Substitute this inequality into \eqref{sigmaval} to obtain
$$
	\alpha^2\le {\sigma^2\over \gamma}
	\Big\{
		\log(C_1/\sigma^2)+\log\Big(
		{\log(C_1/\sigma^2)+\kappa\over\log(1/\theta)}\Big)\Big\}.
$$
Defining
$$
	f(\sigma)=
        \log(C_1/\sigma^2)+\log\Big(
        {\log(C_1/\sigma^2)+\kappa\over\log(1/\theta)}\Big)
\words{and}
	C_0={3c\over \gamma}=12c\|\Psi\|
$$
then finishes the proof.
\end{proof}

\section{Comparison with Simulation}\label{compute}

The purpose of this section is to compare the analytical results
from Section~\ref{nudgemod} to numerical simulation, 
note areas of interest for future work and make some concluding remarks.
To this end we consider a natural class of 
type-I interpolants obtained from noisy observational measurements 
given by local spatial averages taken near distinct points in space.
We then choose numerical values of $h$, $\delta$ and $\mu$ that
{\rev do not} necessary satisfy the conditions of Proposition~\ref{improvedpath}
but work in practice.
Note that the existence of such optimistic 
choices for $h$, $\delta$ and $\mu$
is guided by previous computational experience,
see for example \cite{Olson2003}, \cite{Olson2008} and \cite{Gesho2016},
and stems from 
the fact that even the simplest
analytic bounds which determine $\rho_H$ and $\rho_V$ appear orders
of magnitude too large.
{\rev Moreover, the existing proof techniques rely essentially on 
the dissipation even though convection appears to play an important
physical role in generating the small scales from the large.}

Before beginning we remark that all computations were performed 
on a {\rev $512\times 512$}
spatial grid using a fully-dealiased pseudo-spectral method with time 
steps of size $dt={\rev .015625}$ carried out
by means of the fourth-order exponential time integrator proposed 
by Cox and Matthews~\cite{Cox2001} and regularized 
as suggested by Kassam and {\rev Trefethen}~\cite{Kassam2005}.

Let $d(x,y)$ be distance in the $2\pi$-periodic domain
$\T$ given by
$$d(x,y) = \min \big\{\, |x-y+2\pi m| : m\in\Z^2\,\big\}.$$
Choose $x_i\in\T$ for $i=1,\ldots,N$ 
and for $r>0$ fixed consider the noise-free measurements $m_j$
at time $t$ about those locations given by
$$
	m_j(t)=
		{1\over\pi r^2}
		\int_{d(y,x_j)<r} U(t,y)dy.
$$
We remark that the integrals over balls of radius 
$r$ in the definition of $m_j$ represent the
fact that physical instrumentation always measures a local average of
the velocity near a point rather than the exact velocity at a single
point in space.  Mathematically, this averaging provides a regularizing
effect proven in Appendix~\ref{localaverages} 
that leads to a type-I interpolant observable.  In our
numerics we take $r=2\pi \sqrt{6}/{\rev 512}\approx {\rev 0.0300598}$.  This corresponds
to an average of 21 points from the spatial grid per measurement.

\begin{figure}[h!]
    \centerline{\begin{minipage}[b]{0.75\textwidth}
    \caption{\label{interpolant}%
    The left shows the Voronoi tessellation resulting from $81$ 
    points arranged in a rectangular grid; 
    the right illustrates the same same number of points
    chosen at random.
	The circles depict the discs 
    over which the local averages are taken that represent
    the observational measurements of the velocity field.}
    \end{minipage}}
    \centerline{
		\includegraphics[height=0.35\textwidth]{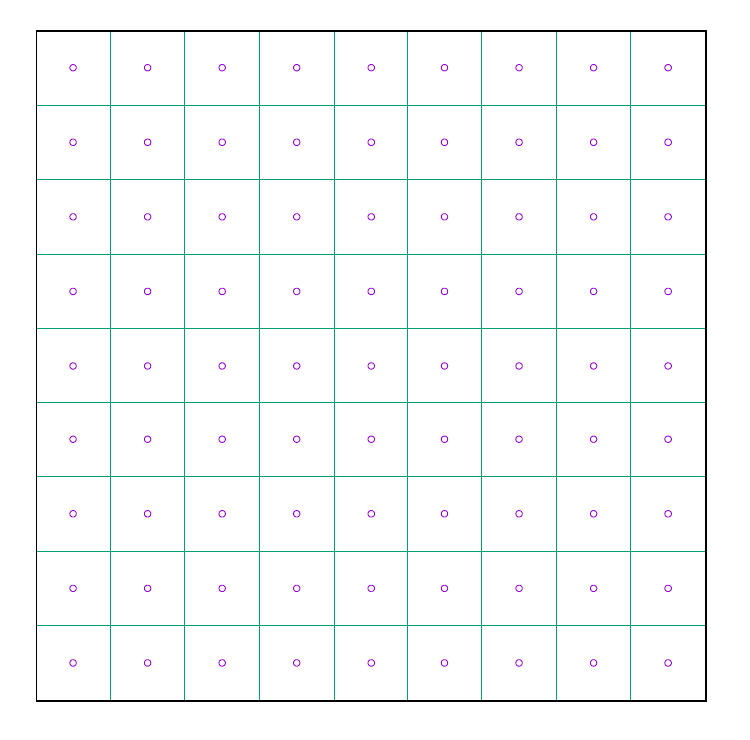}
		\includegraphics[height=0.35\textwidth]{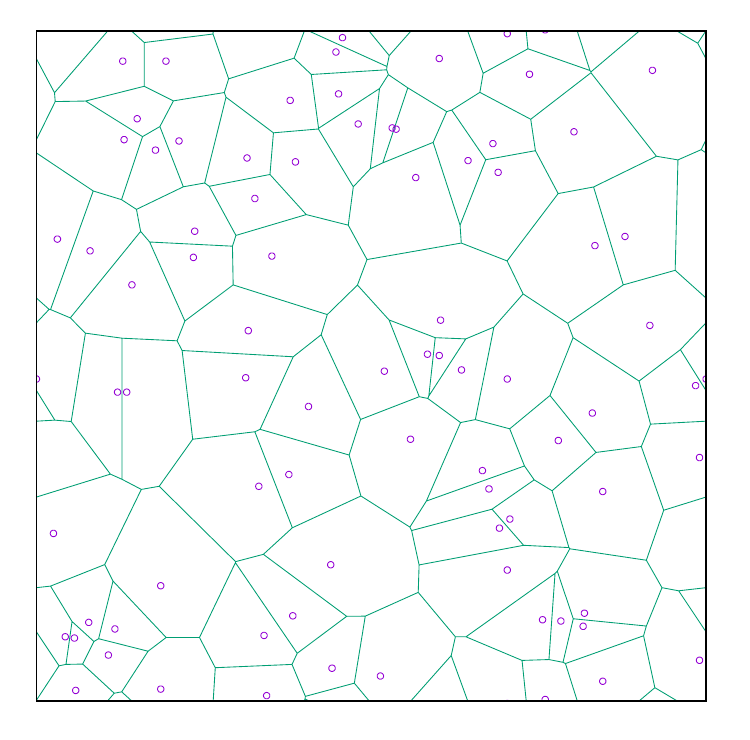}
	}
\end{figure}

Interpolate the measurements $m_j$ every $\delta$ units in time using 
simple a
piecewise-constant interpolant $I_h$ onto the Voronoi 
tessellation given by the points $x_i$.
Thus,
\begin{equation}\label{noisefree}
I_h U(t_n,x)=\sum_{j=1}^N \chi_j(x)m_j(t_n)
\wwords{where} t_n=\delta n
\end{equation}
and $\chi_j$ is the characteristic function given by
\begin{plain}$$
	\chi_j(x)=\cases{
		1 & for $d(x,x_j)\le\min\big\{\,d(x,x_i):i=1,\ldots,N\,\big\}$\cr
		0 & otherwise.
	}
$$\end{plain}%

On the left in Figure \ref{interpolant} is a depiction of the tessellation 
when the points $x_i$ lie on a regular grid.
The right illustrates points chosen randomly.
In the present paper we take $\delta=1$ and focus on the 
regular $9\times 9$ grid with $N=81$
while noting when the points $x_i$ do not lie on such a grid that 
convergence of the approximating solution to the free-running
solution can be more erratic.
{\rev Similar effects have been observed by
Desamsetti, Dasari, Langodan, Titi, Knio and 
Hoteit~\cite{Desamsetti2019} in the context of an operational
mesoscale weather-prediction system
and also by Carlson, Van Roekel, Godinez, Petersen and 
Larios~\cite{Carlson2021}
for an ocean model test case that simulates a wind-driven
double gyre.
Further study of observational data that exhibits areas of 
low-resolution spatial measurements along with localized
areas of high-resolution measurements}
is planned for a future work.
We remark that the related situation where all 
observations as well as the effects of 
the nudging are confined to a subdomain of $\T$ was considered 
by Biswas, Bradshaw and Jolly in \cite{Biswas2021}.

Define the length scale $h$ by
$$h=\max_{x\in\T} \min\big\{\, d(x,x_i): i=1,\ldots,N\,\big\}$$
where intuitively $1/h$ is the minimum observational density.
For a regular $9\times 9$ grid one has $h=\pi\sqrt{2}/9\approx 0.49365$;
however, due to the {\rev $512\times 512$} spatial 
discretization of the domain $\T$ our numerics actually
satisfied $h={\rev 56\pi\sqrt{2}/512}\approx 0.48594$.

Under the assumptions stated in Theorem~\ref{thmintobs} 
there exists a constant $c_0$ such that
$$
	\|\Phi-I_h \Phi\|_{L^2}^2 \le c_0 h^2 \|\Phi\|^2\wwords{for all} \Phi\in V.
$$
We remark for the values of $r$ and $h$ used in our simulations that
$\gamma=r/h\approx {\rev 0.061859}$ and $h=\tilde h$ implies taking 
$c_0=2\cdot3^6/(\pi\gamma^2) \approx {\rev 121284}$ is sufficient.  This notably large number demonstrates
one of the limitations in using small spatial averages to obtain
a type-1 interpolant observable and will play a role shortly when it
comes to filtering the outliers.

Although setting $J_h=P_H I_h$ would satisfy Definition~\ref{interpolants}
and result in a type-I interpolant observable,
we further include a spatial smoothing filter 
to remove the high-frequency spillover
that would otherwise result from the discontinuities in $I_h$.
The importance of such spatial smoothing was demonstrated computationally 
in~\cite{Gesho2016} and further employed in~\cite{Celik2019} for
the analysis of a different discrete-in-time data-assimilation algorithm.

In this paper we employ a smoothing filter 
given by the spectral projection
$$
    P_\lambda U=\sum_{|k|^2\le\lambda} U_k e^{ik\cdot x}
\wwords{where}
	U=\sum_{k\in\Z^2} U_k e^{ik\cdot x}
$$
and define $J_h=P_\lambda P_H I_h$.
One advantage of using $P_\lambda$ in our
interpolant observable
rather than a different spatial filter
is the simple way the improved 
Poincar\'e inequality
$$
	\lambda |(I-P_\lambda) \Phi|^2 \le\|(I-P_\lambda) \Phi\|^2
$$
along with the Pythagorean theorem
$$
	|\Phi-J_h\Phi|^2=|(I-P_\lambda)\Phi|^2+
		|P_\lambda (\Phi - P_H I_h\Phi)|^2
		\le (\lambda^{-1} + c_0h^2) \|\Phi\|^2,
$$
{\rev and a fixed constant $c_3$}
yields a smooth type-I
interpolant observable with {\rev $c_1=c_3+c_0$}
provided $\lambda^{-1}\le{\rev c_3}h^2$.
{\rev From a practical point of view it is reasonable for 
the resolution of the spectral filter
and the interpolant observable to be comparable.}

In our numerics we take $\lambda=80$.
After accounting for the conjugate symmetry $U_k=-\overline{U_{-k}}$
in the Fourier transform of a real vector field,
$$
	{\rm card}\big\{\,k\in\Z^2:0<|k|^2\le 80\,\}
	=248
$$
implies $P_\lambda$ consists of $124$ independent Fourier modes.  
This is similar in quantity to the~$81$ local averages 
which appear in the unfiltered interpolant $I_h$ of the $9\times 9$ spatial grid.

Having thus described $J_h$ we now detail the noise process 
that leads to the noisy interpolant $\widetilde J_h U(t_n)$.
Suppose at each time $t_n$ the measurements $m_j(t_n)\in \R^2$ are
contaminated by Gaussian errors such that
$$
	\widetilde m_j(t_n)=m_j(t_n)+\varepsilon_{2j-1} Y_{2j-1,n}
		+\varepsilon_{2j}Y_{2j,n}
\wwords{for} j=1,\ldots,N
$$
where 
$\varepsilon_j\in\R^2$ and
$Y_{j,n}$ form a family of independent standard normal
random variables.

Upon setting
$$
	\widetilde J_h U(t_n) = P_\lambda P_H \widetilde I_h U(t_n)
\wwords{where}
	\widetilde I_h U(t_n,x)=\sum_{j=1}^N \chi_j(x)\widetilde m_j(t_n)
$$
we obtain that
$$
	\eta_n=\widetilde J_h U(t_n)- J_h U(t_n)
		= P_\lambda P_H \sum_{j=1}^N \chi_j(\varepsilon_{2j-1} Y_{2j-1,n}
        +\varepsilon_{2j}Y_{2j,n})
		= \sum_{i=1}^{2N} \sigma_i Y_{i,n} \psi_i,
$$
where
$$
	\sigma_i= |P_\lambda P_H \chi_{\lfloor (i+1)/2\rfloor}\varepsilon_i|
\wwords{and}
\psi_i=
	P_\lambda P_H \chi_{\lfloor (i+1)/2\rfloor}\varepsilon_i/\sigma_i
$$
for $i=1,\ldots,2N$
and $\lfloor (i+1)/2\rfloor$ denotes
the greatest integer less than or equal to $(i+1)/2$.
Thus, $\eta_n$ has the form
hypothesized in \eqref{etadist}.

In our numerics we take 
\begin{plain}\begin{equation}
	\varepsilon_i={\varepsilon\over 2\pi\sqrt 2}\cases{
		e_1& for $i$ even\cr
		e_2& for $i$ odd
	\cr}
\end{equation}\end{plain}%
where $e_1=(1,0)$, $e_2=(0,1)$ 
and $\varepsilon>0$.
Note the $2\pi\sqrt 2$ normalization was chosen so the noise 
in unfiltered interpolant $\widetilde I_h$ satisfies
$\V[\widetilde I_h]=\varepsilon^2$.
Moreover,
$$
	\sigma^2= \sum_{i=1}^{2N} \sigma_i^2=
		\varepsilon^2 
		\sum_{j=1}^{N} \big(
		|P_\lambda P_H \chi_j e_1|^2+
		|P_\lambda P_H \chi_j e_2|^2\big)
	\approx (0.40058) \varepsilon^2
$$
shows $\sigma^2$ is proportional to $\varepsilon^2$
which we shall vary as $\varepsilon=10^{-\ell}$ for $\ell=4,\ldots,8$
as a test on the bounds in Theorem~\ref{mnudge}.

\begin{figure}[h!]
    \centerline{\begin{minipage}[b]{0.75\textwidth}
    \caption{\label{initialcond}%
	The left shows level-curves of $\curl f$ {\rev spaced $0.03$ apart
	with positive in red, negative in blue and zero in black}
	where $f$ is the 
    time-independent body forcing used in all the simulations.
    The right shows level-curves of the vorticity $\curl U_0$
	{\rev spaced $0.2$ apart}
    where $U_0$ is the initial condition of the 
	free-running solution.}
    \end{minipage}}
    \centerline{
		\includegraphics[height=0.35\textwidth]{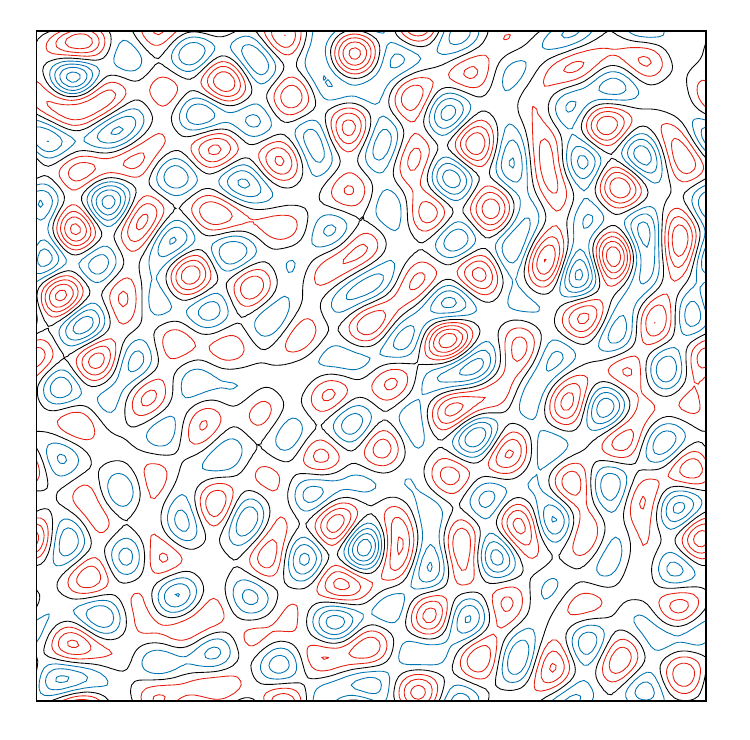}
		\includegraphics[height=0.35\textwidth]{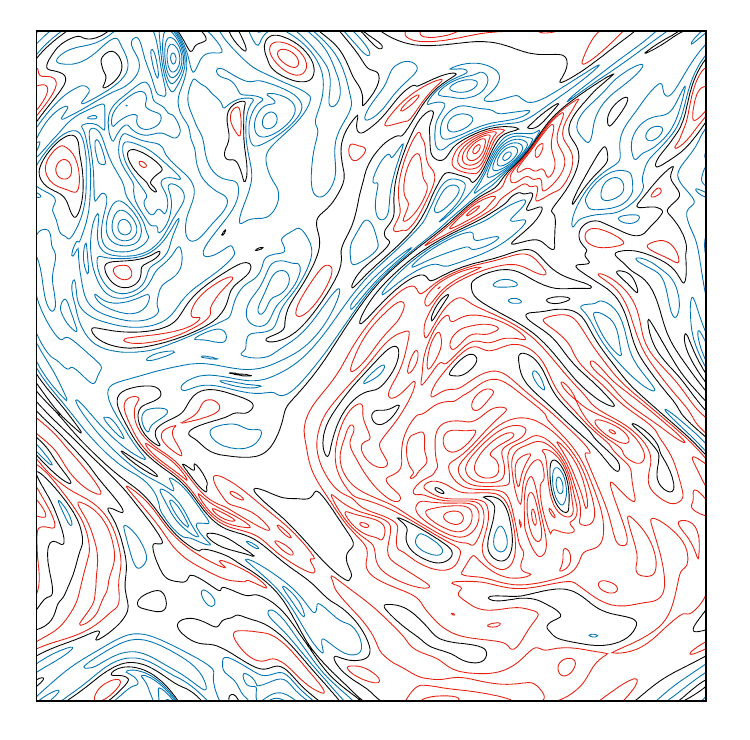}
	}
\end{figure}

The time-independent body forcing $f$ used for our simulations
satisfies
\begin{equation}\label{dforce}
	f=\sum_{\lambda_m\le |k|^2\le \lambda_M} f_k e^{ik\cdot x}
\wwords{with}
	k\cdot f_k=0,
\end{equation}
where $\lambda_m=100$ and $\lambda_M=142$. Here the Fourier coefficients
$f_k$ are time independent.  They were obtained by projecting randomly 
chosen values onto $H$ and then rescaling them to achieve the $L^2$ 
norm $|f|=0.025$.  Upon setting the viscosity $\nu=0.0001$ we further
have that the Grashof number of the free-running solution satisfies
$$
	{\bf Gr}(f)={1\over {\rev \lambda_1}\nu^2}|f|=2.5\times 10^{6}.
$$
{\rev Here $\lambda_1=1$ due to the $2\pi$-periodic
boundary conditions.
Recall also that $|f|$ refers to the $L^2$ norm, which in
terms of the Fourier series may be computed as
$$
	|f|=2\pi\bigg\{\sum_{\lambda_m\le |k|^2\le \lambda_M} 
		|f_{k,1}|^2+|f_{k,2}|^2
	\bigg\}^{1/2}
\wwords{where} f_k=(f_{k,1},f_{k,2}).
$$
}%
Figure \ref{initialcond} illustrates the forcing function used in 
our simulations on the left.
The exact values of the Fourier modes $f_k$
corresponding {\rev to} this force are given in \cite{Olson2008} and are known
to lead to a complicated time-dependent flow $U$.

The initial condition $U_0$ for the reference solution is 
theoretically assumed to lie on the global attractor.
Numerically, we take $U_0$ as the final state of a long-time 
integration of~\eqref{ns2d} starting
from {\rev zero at time $t=-20480$} in the past.
{\rev Note the forcing function given by \eqref{dforce} is sufficient
such that by time $t=0$ all modes such that $k\ne 0$ which satisfy
the $2/3$ antialiasing condition are present in the solution.
Furthermore, the energetics of the
flow appear to represent the long-time statistical 
behavior that results from the forcing.}
The graph on the right in Figure \ref{initialcond} depicts 
the initial condition for the
reference trajectory used for our data-assimilation experiments.

For comparison with our theory, we note 
the estimates in \cite{Olson2003} applied to 
the forcing given by \eqref{dforce} lead to the
{\it a priori\/} bounds 
$$
	\rho_H
\le {1\over \lambda_m^{1/2}} \rho_V = 25
\wwords{and }
	\rho_V\le {1\over \lambda_1^{1/2} \nu} |f| = 250.
$$
Therefore, a suitable theoretical bound to filter the outliers satisfies
\begin{plain}$$
\eqalign{
	M&=\rho_H+c_1^{1/2}h\rho_V\approx 25+ c_1^{1/2} (121.49)
		\approx {\rev 42333}.
}$$
\end{plain}%
Setting $M={\rev 42333}$ results in essentially
no outliers being filtered by our modifications to the original
time-delay nudging method.  
The size of $M$ arises directly from our previous
estimate on $c_1$ along with values of $\rho_H$ and $\rho_V$
that, as already mentioned, appear much larger than needed.
The possibility of empirically tuning this cutoff to achieve better 
numerical results is interesting but outside the scope of the
present paper.

To compute the expectation $\E\big[\|U-u\|^2\big]$ 
whose bounds are the subject of Theorem~\ref{mnudge} 
we consider the time evolution of an ensemble of 500 approximating
solutions obtained from observational measurements of
the same reference solution $U$ each contaminated by
different realizations of the Gaussian noise $\eta_n$ 
parameterized by $\omega_j\in\Omega$ where $\Omega$ is
a suitable probability space.
Thus, $u(t;w_j)$ is a random point in $V$ and we approximate
$$
	\E\big[\|U(t)-u(t)\|^2\big]\approx
	{1\over 500} \sum_{j=1}^{500} \|U(t)-u(t;\omega_j)\|^2.
$$

\begin{figure}[h!]
    \centerline{\begin{minipage}[b]{0.75\textwidth}
    \caption{\label{eassim}%
	Evolution of $\E\big[\|U-u\|^2\big]$ for a fixed
    solution $U$ with approximating solutions $u$
	obtained by observational measurements contaminated by
	Gaussian noise where $\varepsilon\in \{ 10^{-4},\ldots, 10^{-8}\}$.
	The shaded regions illustrate the bounds $I_p$ on 
	$p=88$, $70$ and $40$ percent
	of the 500 paths in each ensemble.
	The $\varepsilon=10^{-5}$ and $10^{-7}$ trajectories have been omitted
	for $t\in[2000,2500]$ to avoid overlap.}
    \end{minipage}}
    \centerline{
		\includegraphics[height=0.45\textwidth]{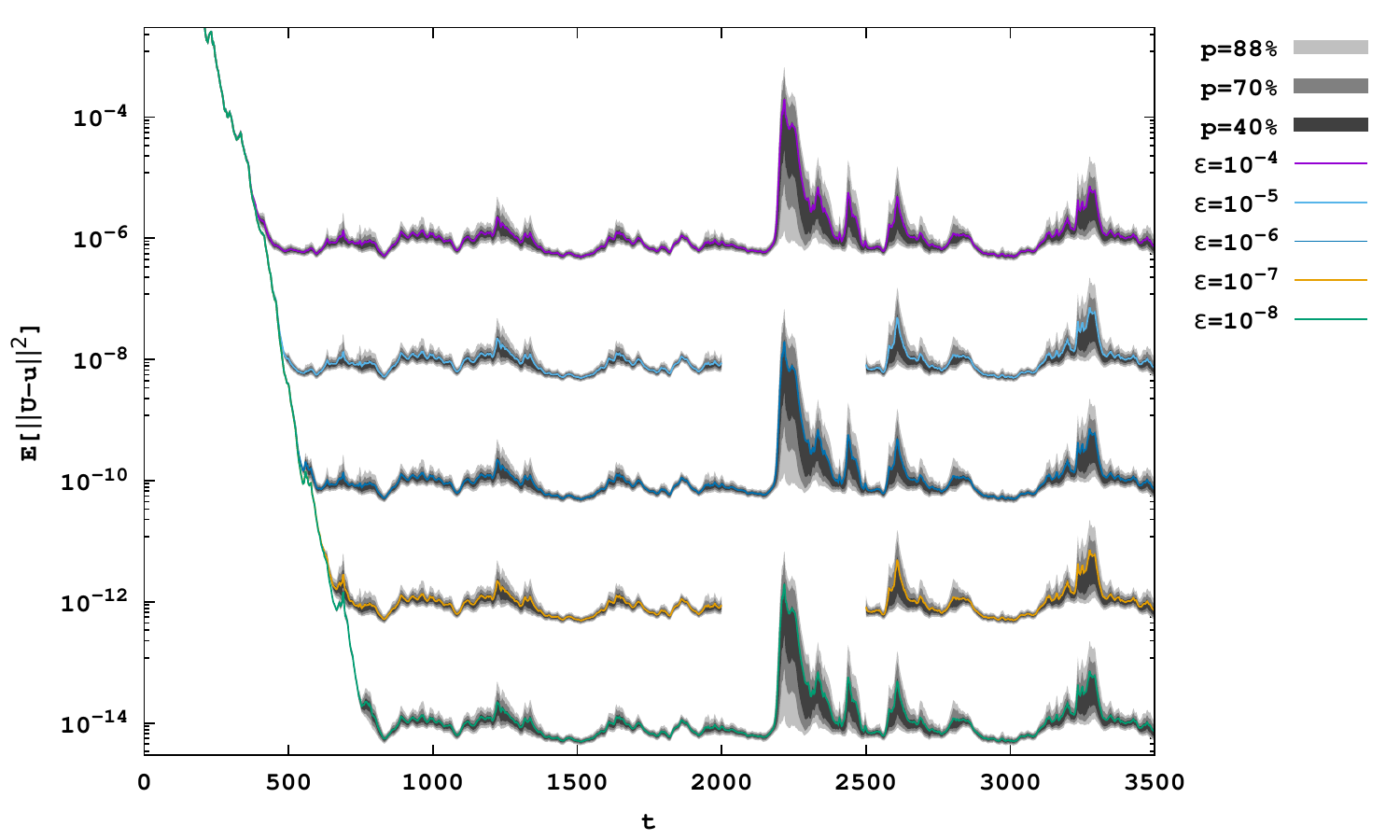}
	}
\end{figure}

The statistics of the pathwise trajectories may
further be characterized by
the intervals $I_p=[a,b]$ for $p\in[0,1]$ where
${\bf P}\big\{ \|U-u\|^2\in I_p \big\} = p$ by
choosing $a$ and $b$ such that
$$
     {\bf P}\big\{ \|U-u\|^2\ge a \big\} = (1+p)/2
\wwords{and}
     {\bf P}\big\{ \|U-u\|^2\le b \big\} = (1+p)/2.
$$
At any point in time approximations for $a$ and $b$ 
are given by the thresholds below and above which 
the desired percentage of trajectories in the ensemble lie.

Parts of the time evolution of 
$\E\big[\|U-u\|^2\big]$ 
for selected values of $\sigma^2=(0.40058)\varepsilon^2$ 
is depicted in Figure~\ref{eassim} 
along with a shaded region that indicates
the intervals $I_p$ corresponding to $p=88$, 60 and 30.
Upon characterizing the average large-eddy turnover time in the
reference solution by
$$
	\tau=
	{4\pi^2\over T}\int_0^T 
		\|U\|_{H^{-1/2}(\T)}^2\Big/ 
	\Big({1\over T}\int_0^T
		\|U\|_{L^{2}(\T)}^2\Big)^{3/2}
	\approx {\rev 30.585},
$$
we note the full trajectory of the ensemble was computed
using {\rev $655360$} time steps 
until time $T={\rev 10240}$ or equivalently for $T/\tau\approx {\rev 334.8}$ 
large-eddy turnovers.

We remark by $t=1500$
sufficient time has passed 
over which the initial condition
$u_0=0$ is forgotten and after which the expected 
value of $\E\big[\|U-u\|^2\big]$ starts fluctuating about its mean.
The maximum bound for $t\in[1500,T]$ 
is determined by the interval $t\in [2000,2500]$ 
illustrated in Figure~\ref{eassim}.
Since the same reference solution $U$ was used for each run,
it is only a little surprising that this maximum occurs 
in the same place for the different choices of $\varepsilon$.
More notable and
further illustrated in Figure~\ref{easlope}
is the fact that this maximum
is about {\rev 100} times larger than the time-averaged value
taken over the same interval.

Computations not reported here suggest even greater excursions from 
the average can happen when the measurement points $x_i$ are not 
given by a uniform grid.  
At the same time, simulations performed by 
Law, Sanz--Alonso, Shukla and Stuart \cite{Law2016} for 
the Lorenz 96 model suggest that a careful placement of the
points $x_i$ could yield more accurate synchronization
between $U$ and $u$ than a uniform grid.
Improved results for observations which scan the domain
over time were obtained by Larios and Victor
\cite{Larios2021} 
for the Chafee--Infante equations,
{\rev by Franz, Larios and Victor \cite{Franz2022} for the
two-dimensional Navier--Stokes equations}
and by Biswas, Bradshaw and Jolly \cite{Biswas2021} 
{\rev see also \cite{Biswas2023}}
for the two-dimensional Navier--Stokes equations.
This leads us to an interesting question for future investigation:
If the measurement points themselves
are advected by the flow over time like buoys in the ocean, could 
the resulting distribution of the $x_i$ be better than a uniform grid?

\begin{figure}[h!]
    \centerline{\begin{minipage}[b]{0.75\textwidth}
    \caption{\label{easlope}%
	Dependency of bounds on $\E\big[\|U-u\|^2\big]$
	for $t\in [1500,{\rev 10240}]$ as a function of $\sigma^2$.
	The values denoted by {\it max\/} are numerical versions
    of the analytic bounds provided by Theorem \ref{mnudge} and
    depend on $\sigma^2$ in a similar way. 
	Values for {\it avg\/} and {\it min\/} are shown for comparison.
	}
    \end{minipage}}
    \centerline{
		\includegraphics[height=0.45\textwidth]{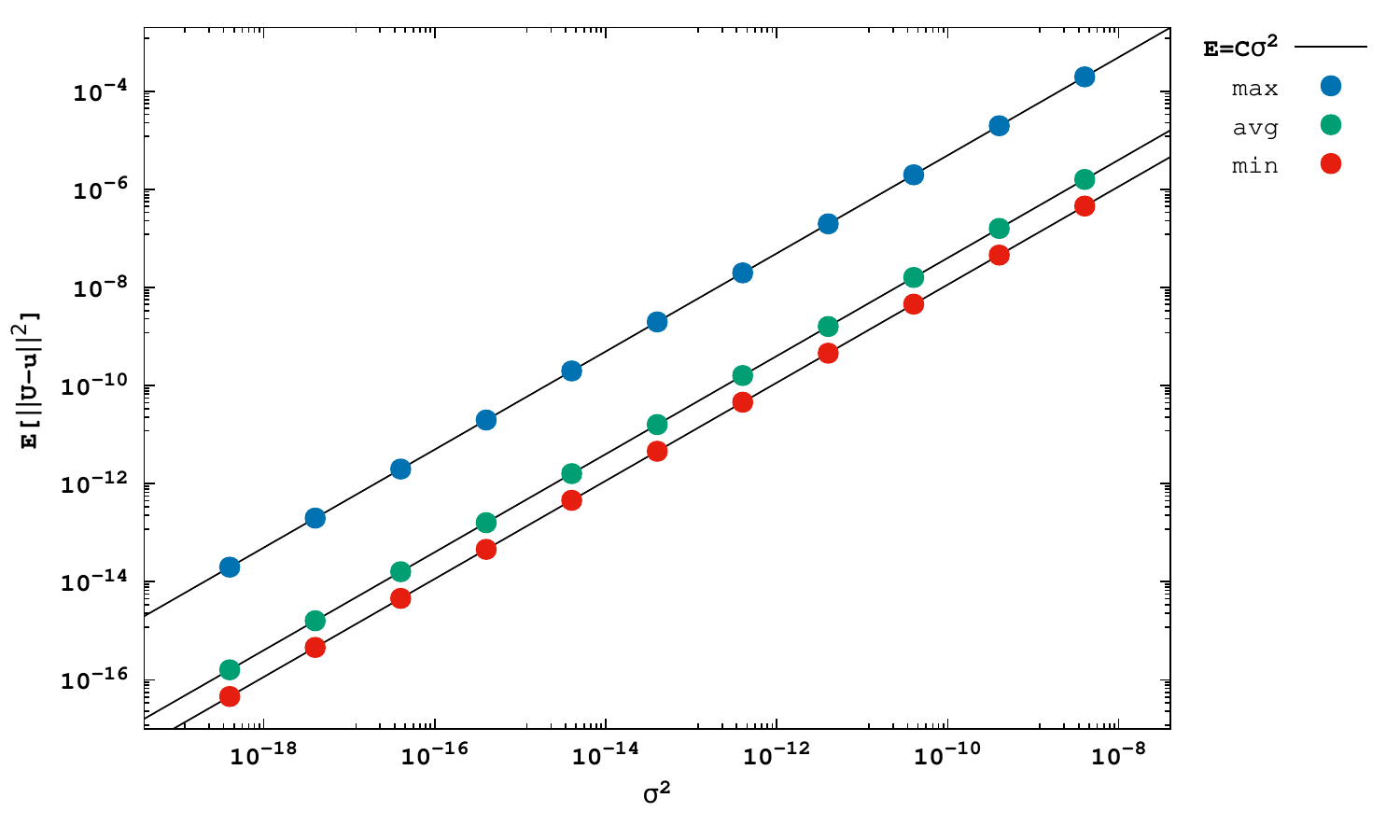}
	}
\end{figure}

In either case, the {\rev 100}-fold difference between the maximum of 
$\E\big[\|U-u\|^2\big]$ and the time-averaged value observed in the
present simulations may be due to the fact that
our numerical choices for $h$, $\delta$ and $\mu$ are
significantly more optimistic than guaranteed by the theory.
In particular, if $h$ and $\delta$ were smaller {\rev with $\mu$ 
and perhaps $r$} larger, it
is reasonable to suppose such large fluctuations
would not happen over time.

We conclude by noting the graph displayed as Figure~\ref{easlope} shows
that the maximum, minimum and average values of 
$\E\big[\|U-u\|^2\big]$ are all proportional to $\sigma^2$ for over 10 decimal 
orders of magnitude.  This suggests,
except for the constants and 
the logarithmic correction,
that the qualitative
behavior of the analytic bound in Theorem~\ref{mnudge}
is physically reasonable.

\appendix

\section{Local Spatial Averages}
\label{localaverages}

In this appendix we show the interpolant observable $I_h$ 
defined by \eqref{noisefree} is a type-I interpolant observable.
Mathematically this interpolant may be seen as a 
physically reasonable mix of the determining nodes and 
volume elements studied by Jones and Titi in~\cite{Jones1993}.
Figure~\ref{interpolant} illustrates the variety of interpolants treated by 
this construction.  Note that the discs which represent the
spatial averages in our observational measurements overlap 
in the tessellation on the right.  Note also that the maximum 
distance $h$ between 
any point in the domain and the nearest point of observation
is much less for the regular grid.

Having examined two examples of the interpolants under
consideration, we now prove
\begin{theorem}\label{thmintobs}
Suppose $r$ is proportional to $h$ such that
$r=\gamma h$ for some $\gamma\in(0,1)$.  Then the piecewise
constant interpolant $I_h\colon V\to L^2(\T)$ 
defined by \eqref{noisefree} satisfies
$$
	\|\Phi-I_h(\Phi)\|_{L^2(\T)}^2
		\le c_0 h^2 \|\Phi\|_{H^1(\T)}^2
\wwords{for all} \Phi\in V,
$$
where
\begin{equation}\label{c1const}
	c_0={2^53^6\over \pi\gamma^2} \wwords{and}
	h=\sup_{x\in\T} \min\{\, d(x,x_i): i=1,\ldots,N\,\}.
\end{equation}
\end{theorem}
\begin{proof}
Without loss of generality, we shall assume $h=2\pi/\kappa$ for
some $\kappa\in\N$.  It this were not the case,
there would be $\kappa\in\N$ such that
$$
{2\pi\over \kappa+1}< h \le {2\pi\over \kappa}.$$
Then, upon replacing $h$ by $\tilde h=2\pi/\kappa$ in the
following proof and
noting $\tilde h\le (\kappa+1)h/\kappa\le 2h$, 
the desired result for any $h$ could be obtained by
a simple modification of the constant $c_0$.
We note before proceeding, that the value of $c_0$ stated in 
\eqref{c1const} has, in fact, already been modified by a factor 
of 16 to take the general case into account.

Now, divide the torus $\T$ into $M=\kappa^2$ equal squares
with sides of length $h$ denoted
$$
	R_{pq}=[(p-1)h, ph)\times[(q-1)h,qh)
\wwords{for} p,q=1,\ldots,\kappa.
$$
Given $\Phi\colon\T\to\R^2$ extend $\Phi$ by periodicity to all of $\R^2$.
For convenience, we shall continue to call the extended function $\Phi$
and assume that $\Phi(x+2\pi m)=\Phi(x)$
for every $x\in\R^2$ and $ m\in\Z^2$.
Having done this, we define
$$
{\cal U}_j=\big\{\, x\in\R^2 : |x-x_j| \le
		\min\{\, d(x,x_i) : i=1,\ldots, N\,\}\big\}$$
and note that 
$$
	I_h(\Phi)(\xi,\eta)={1\over \pi r^2} 
		\int_{B_r(x_j)} \Phi(s,t) ds\,dt
\wwords{for every}
	(\xi,\eta)\in {\cal U}_j.
$$
Thus, after periodic extension, any integral over 
$\T$ is equal to the same integral over
$\bigcup_{j=1}^{m} \, {\cal U}_j$.

Next, define ${\cal J}_{pq}=\{\,j:x_j\in R_{pq}\,\}$.
Since the $R_{pq}$ are disjoint, then ${\cal J}_{pq}$'s form
a partition of $\{1,\ldots,n\}$.
Expand the $R_{pq}$ by $h$ in each dimension to obtain
$$
	Q_{pq}=[(p-2)h, (p+1)h)
		\times[(q-2)h,(q+1)h)
\words{for}
		p,q=1,\ldots,\kappa.
$$
It follows that 
${\cal U}_j\times B_r(x_j)\subseteq Q_{pq}^2$
for each $j\in {\cal J}_{pq}$.
Consequently,
\begin{align*}
	\|\Phi-I_h \Phi\|_{L^2(\T)}^2
	&=\int_\T
		\big|\Phi(\xi,\eta)-I_h(\Phi)(\xi,\eta)\big|^2d\xi,d\eta\\
	&=\sum_{j=1}^n\int_{{\cal U}_j} \Big|\Phi(\xi,\eta)-
		{1\over \pi r^2} \int_{B_r(x_j)}
        \Phi(s,t) ds\,dy
	\Big|^2d\xi\,d\eta\\
	&={1\over \pi^2r^4}
		\sum_{j=1}^n\int_{{\cal U}_j} \Big|
		\int_{B_r(x_j)}
		\big(\Phi(\xi,\eta)-
        \Phi(s,t)\big) ds\,dt
	\Big|^2d\xi\,d\eta\\
	&\le {1\over \pi r^2}
		\sum_{j=1}^n
		\int_{{\cal U}_j\times B_r(x_j)}
		\big|\Phi(\xi,\eta)-\Phi(s,t)\big|^2 ds\,dt\,d\xi\,d\eta\\
	&= {1\over \pi r^2}
		\sum_{p,q=1}^\kappa
		\sum_{j\in{\cal J}_{pq}}
		\int_{{\cal U}_j\times B_r(x_j)}
		\big|\Phi(\xi,\eta)-\Phi(s,t)\big|^2 ds\,dt\,d\xi\,d\eta\\
	&\le {1\over \pi r^2}
		\sum_{p,q=1}^\kappa
		\int_{Q_{pq}^2} 
		\big|\Phi(\xi,\eta)-\Phi(s,t)\big|^2 ds\,dt\,d\xi\,d\eta.
\end{align*}
We remark that the last inequality in the estimate above holds 
because the sets ${\cal U}_j$ are disjoint except for a set of measure 
zero, which implies the same for ${\cal U}_j\times B_r(x_j)$.

Since $(\xi,\eta)\in Q_{pq}$ and $(s,t)\in Q_{pq}$, we continue
to estimate as
\begin{align*}
	\big|\Phi(\xi,\eta)-\Phi(s,t)\big|
		&\le \big|\Phi(\xi,\eta)-\Phi(s,\eta)\big|
			+\big|\Phi(s,\eta)-\Phi(s,t)\big|\\
		&=\Big|\int_s^\xi \Phi_x(\alpha,\eta)d\alpha\Big|
		+\Big|\int_t^\eta \Phi_y(s,\beta)d\beta\Big|\\
		&\le 
		\int_{(p-2)h}^{(p+1)h}\big|\Phi_x(\alpha,\eta)\big|d\alpha
		+\int_{(q-2)h}^{(q+1)h}\big|\Phi_y(s,\beta)\big|d\beta.
\end{align*}
This implies
\begin{align*}
	\big|\Phi(\xi,\eta)-\Phi(s,t)\big|^2
	&\le
	2\Big(\int_{(p-2)h}^{(p+1)h}\big|\Phi_x(\alpha,\eta)\big|d\alpha\Big)^2
	+2\Big(\int_{(q-2)h}^{(q+1)h}\big|\Phi_y(s,\beta)\big|d\beta\Big)^2\\
	&\le
	6h\int_{(p-2)h}^{(p+1)h}\big|\Phi_x(\alpha,\eta)\big|^2d\alpha
	+6h\int_{(q-2)h}^{(q+1)h}\big|\Phi_y(s,\beta)\big|^2d\beta.
\end{align*}
It follows that
\begin{align*}
	\int_{Q_{pq}^2}
	&\big|\Phi(\xi,\eta)-\Phi(s,t)\big|^2 d\xi\,d\eta\,ds\,dt\\
	&\le
		\int_{(p-2)h}^{(p+1)h}\!\!\!
		\int_{(q-2)h}^{(q+1)h}\!\!\!
		\int_{(p-2)h}^{(p+1)h}\!\!\!
		\int_{(q-2)h}^{(q+1)h}\!\!\!
		\big|\Phi(\xi,\eta)-\Phi(s,t)\big|^2 d\xi\,d\eta\,ds\,dt\\
	&\le
		6h\Big(
		\int_{(p-2)h}^{(p+1)h}\!\!\!
		\int_{(q-2)h}^{(q+1)h}\!\!\!
		\int_{(p-2)h}^{(p+1)h}\!\!\!
		\int_{(q-2)h}^{(q+1)h}\!\!\!
		\int_{(p-2)h}^{(p+1)h}\big|\Phi_x(\alpha,\eta)\big|^2
			d\xi\,d\eta\,ds\,dt\,d\alpha\\
	&\phantom{6h\Big(}+
		\int_{(p-2)h}^{(p+1)h}\!\!\!
		\int_{(q-2)h}^{(q+1)h}\!\!\!
		\int_{(p-2)h}^{(p+1)h}\!\!\!
		\int_{(q-2)h}^{(q+1)h}\!\!\!
		\int_{(q-2)h}^{(q+1)h}\big|\Phi_y(s,\beta)\big|^2
			d\xi\,d\eta\,ds\,dt\,d\beta\Big)\\
	&=6h\Big(27h^3
		\int_{(q-2)h}^{(q+1)h}\!\!\!
		\int_{(p-2)h}^{(p+1)h}\big|\Phi_x(\alpha,\eta)\big|^2
			d\eta\,d\alpha\\
	&\phantom{6h\Big(}+27h^3
		\int_{(p-2)h}^{(p+1)h}\!\!\!
		\int_{(q-2)h}^{(q+1)h}\big|\Phi_y(s,\beta)\big|^2
			ds\,d\beta\Big)\\
	&\le 162 h^4\big(
		\|\Phi_x\|_{L^2(Q_{pq})}^2
		+\|\Phi_y\|_{L^2(Q_{pq})}^2
		\big).
\end{align*}
Therefore,
\begin{align*}
    \|\Phi-I_h\Phi\|_{L^2(\T)}^2
		&\le {164 h^4\over\pi r^2}
		\sum_{p,q=1}^\kappa \big(
		\|\Phi_x\|_{L^2(Q_{pq})}^2
		+\|\Phi_y\|_{L^2(Q_{pq})}^2
		\big)\\
		&= {164 h^4\over\pi r^2}
		\sum_{p,q=1}^\kappa \big(
		9\|\Phi_x\|_{L^2(R_{pq})}^2
		+9\|\Phi_y\|_{L^2(R_{pq})}^2
		\big)
		={1458 h^4\over\pi r^2} \|\Phi\|_V^2.
\end{align*}
To identify the constant $c_0$ and finish the proof,
substitute $r=h/\gamma$ and, as mentioned in the first
paragraph, replace $h$ by $2h$ to account for general $h$.
\end{proof}

\par\medskip
\noindent
{\bf Data Availability.}  
{\rev Although different computer architectures, software 
libraries and levels of optimization may lead to
variations in rounding errors and a reference trajectory
that follows a different path, 
the expected error of the approximating solution depicted 
in Figure~\ref{easlope} 
and related statistics
are reproducible given the 
information in Section~\ref{compute}.
}
\par\medskip
\noindent
{\rev {\bf Acknowledgements.}  We thank the 
anonymous referees for their careful reading and comments.
These greatly improved the present manuscript.}

\bibliographystyle{abbrv}
\end{document}